\documentclass[oneside]{article}
\usepackage{etex}
\usepackage{amsmath,amssymb,amsthm,graphicx,epsfig,subfigure,float,url}
\usepackage[colorlinks=true]{hyperref}
\usepackage{pdfsync}
\usepackage{color}
\usepackage{esint}
\usepackage{verbatim}
\usepackage[applemac]{inputenc}
\usepackage[usenames,dvipsnames]{pstricks}
\usepackage{pst-grad} 
\usepackage{pst-plot} 
\usepackage{empheq}
\usepackage{tikz}
\usepackage{dsfont}
\usepackage{listings}
\usepackage[usenames,dvipsnames]{xcolor}
\usepackage{enumitem}
\usepackage{mathrsfs}
\usetikzlibrary{decorations.pathreplacing}
\usepackage[makeroom]{cancel}
\usepackage{calligra}
\usepackage{stmaryrd}

\allowdisplaybreaks
\def\R{\textrm{I\kern-0.21emR}}
\def\N{\textrm{I\kern-0.21emN}}

\def\O{{\Omega}}
\def\n{{\nabla}}

\def\e{{\varepsilon}}

\def\uam{{u_{\alpha,m}}}
\def\duamh{{\dot u_{\alpha,m}[h]}}

\def\p{{\varphi}}
\def\pam{{\p_{\alpha,m}}}
\def\tsm{{\tilde \sigma_{\alpha,m}}}
\def\tm{{\tilde m}}
\def\S{{\mathbb S(0,r^*_0)}}
\def\B{{\mathbb B}}
\def\zetat{{\zeta_\alpha(m_t,\Lambda_-(m_t))}}

\def\r{{\frac{r^*_0}R}}

\usepackage{xargs}
 \usepackage[colorinlistoftodos,textsize=small]{todonotes}
 \newcommandx{\unsure}[2][1=]{\todo[linecolor=red,backgroundcolor=red!25,bordercolor=red,#1]{#2}}
 \newcommandx{\change}[2][1=]{\todo[linecolor=blue,backgroundcolor=blue!25,bordercolor=blue,#1]{#2}}
 \newcommandx{\info}[2][1=]{\todo[linecolor=green,backgroundcolor=green!25,bordercolor=green,#1]{#2}}
 \newcommandx{\improvement}[2][1=]{\todo[linecolor=yellow,backgroundcolor=yellow!25,bordercolor=yellow,#1]{#2}}

\renewcommand{\geq}{\geqslant}
\renewcommand{\leq}{\leqslant}

\newtheorem{theorem}{Theorem}

\newtheorem*{thmnonumbering}{Theorem}
\newtheorem*{prnonumbering}{Proposition}

\newtheorem{proposition}{Proposition}
\newtheorem{corollary}{Corollary}
\newtheorem{definition}{Definition}
\newtheorem{lemma}{Lemma}

\theoremstyle{definition}
\theoremstyle{definition}\newtheorem{remark}{Remark}

\newcommand\wo{{W^{1,2}_0(\O)}}

\title{Shape optimization of a  weighted two-phase Dirichlet eigenvalue\footnote{I. Mazari and Y. Privat were partiallly supported by the French ANR Project ANR-18-CE40-0013 - SHAPO on Shape Optimization. I Mazari was partially supported by the Austrian Science Fund (FWF) projects I4052-N32 and F65.  I. Mazari, G. Nadin and Y. Privat were partially supported by the Project "Analysis and simulation of optimal shapes - application to lifesciences" of the Paris City Hall.
}}

\author{Idriss Mazari\footnote{CEREMADE, UMR CNRS 7534, Universit\'e Paris-Dauphine, Universit\'e PSL, Place du Mar\'echal De Lattre De Tassigny, 75775 Paris cedex 16, France (\texttt{mazari@ceremade.dauphine.fr}).}
       \and Gr\'egoire Nadin\footnote{ CNRS, Sorbonne Universit\'es, UPMC Univ Paris 06, UMR 7598, Laboratoire Jacques-Louis Lions, F-75005, Paris, France (\texttt{gregoire.nadin@sorbonne-universite.fr})}
	\and Yannick Privat\footnote{Universit\'e de Strasbourg, CNRS UMR 7501, INRIA, Institut de Recherche Math\'ematique Avanc\'ee (IRMA), 7 rue Ren\'e Descartes, 67084 Strasbourg, France ({\tt yannick.privat@unistra.fr}).}\textsuperscript{~~}\footnote{Institut Universitaire de France (IUF).}
	}

\date{}

\begin{document}

\maketitle

\begin{abstract}

This article is concerned with a spectral optimization problem: in a smooth bounded domain $\O$, for a bounded function $m$ and a nonnegative parameter $\alpha$, consider the first eigenvalue $\lambda_\alpha(m)$ of the operator $\mathcal L_m$ given by $\mathcal L_m(u)= -\operatorname{div} \left((1+\alpha m)\nabla u\right)-mu$. Assuming uniform pointwise and integral bounds on $m$, we investigate the issue of minimizing $\lambda_\alpha(m)$ with respect to $m$. Such a problem is related to the so-called ``two phase extremal eigenvalue problem'' and arises naturally, for instance in population dynamics where it is related to the survival ability of a species in a domain. 
We prove that unless the domain is a ball, this problem has no ``regular'' solution. We then provide a careful analysis in the case of a ball by: (1) characterizing the solution among all radially symmetric resources distributions, with the help of a new method involving a homogenized version of the problem; (2) proving in a more general setting a stability result for the centered distribution of resources with the help of a monotonicity principle for second order shape derivatives which significantly simplifies the analysis.
\end{abstract}
\medskip


\tableofcontents


\section{Introduction and main results}

In recent decades, much attention has been paid to extremal problems involving eigenvalues, and in particular to shape optimization problems in which the unknown is the domain where the eigenvalue problem is solved (see e.g. \cite{Henrot2006,henrot2017shape} for a survey).  
The study of these last problems is motivated by stability issues of vibrating bodies, wave propagation in composite environments, or also on conductor thermal insulation.  

In this article, we are interested in studying a  particular extremal eigenvalues problem, involving a drift term, and which comes from the study of mathematical biology problems; here, we can show that the problem then boils down to a "two-phase" type problem, meaning that the differential operator whose eigenvalues we are trying to optimise has $-\nabla \cdot(A\nabla)$ as a principal part, and that $A$ is an optimisation variable, see Section \ref{BiologicalMotivations}. The influence of drift terms on optimal design problems is not so well understood. Such problems naturally arise for instance when looking for optimal shape design for two-phase composite materials \cite{DambrineKateb,Laurain,MuratTartar}. We expand on the bibliography in Section \ref{SSE:BackgroundH} of this paper, but let us briefly recall that, for composite materials, a possible formulation reads: given $\O$, a bounded connected open subset of $\R^n$ and a set of admissible non-negative densities $\mathcal{M}$ in $\O$, solve the optimal design problem
\begin{equation}\tag{$\hat P_\alpha$}\label{Pg:TwoPhase}
\inf_{m\in \mathcal M} \hat\lambda_\alpha(m)
\end{equation} 
where $\hat \lambda_\alpha(m)$ denotes the first eigenvalue of the elliptic operator 
$$
\hat{ \mathcal L}_\alpha^m:W^{1,2}_0(\O)\ni u\mapsto -\n \cdot \left((1+\alpha m)\n u\right).
$$ 
Restricting the set of admissible densities to {\it bang-bang} ones (in other words to functions taking only two different values) is known to be relevant for the research of structures optimizing the compliance. We refer to Section \ref{Pr:NonExistence} for detailed bibliographical comments.

Mathematically, the main issues regarding Problem \eqref{Pg:TwoPhase} concern the existence of optimal densities in $\mathcal M$, possibly the existence of optimal {\it bang-bang} densities (i.e characteristic functions). In this case, it is interesting to try to describe minimizers in a qualitative way.

In what follows, we will consider a refined version of Problem \eqref{Pg:TwoPhase}, where the operator $\hat{ \mathcal L}_\alpha^m$ is replaced with
\begin{equation}\label{def:Lmalpha}
 \mathcal L_m^\alpha:W^{1,2}_0(\O)\ni u\mapsto -\n \cdot \left((1+\alpha m)\n u\right)-mu.
\end{equation}
Besides its intrinsic mathematical interest, the issue of minimizing the first eigenvalue of $ \mathcal L_m^\alpha$ with respect to densities $m$ is motivated by a model of population dynamics (see Section~\ref{BiologicalMotivations}). 

Before providing a precise mathematical frame of the questions we raise in what follows, let us roughly describe the main results and contributions of this article:
\begin{itemize}
\item by adapting the methods developed by Murat and Tartar, \cite{MuratTartar}, and Cox and Lipton, \cite{CoxLipton}, we show that  the first eigenvalue of $ \mathcal L_m^\alpha$ has no regular minimizer in $\mathcal M$ unless $\Omega$ is a ball;
\item if $\Omega$ is a ball, denoting by $m_0^*$ a minimizer  of $ \mathcal L_m^0$ over $\mathcal M$ (known to be {\it bang-bang} and radially symmetric), we show the following stationarity result: $m_0^*$ still minimizes $ \mathcal L_m^\alpha$ over radially symmetric distributions of $\mathcal M$ whenever $\alpha$ is small enough and in small dimension ($n=1,2,3$). Such a result appears unexpectedly difficult to prove. Our approach is based on the use of a well chosen path of quasi-minimizers and on a new type of local argument.
\item  if $\Omega$ is a ball, we investigate the local optimality of ball centered distributions among all distributions and prove a quantitative estimate on the second order shape derivative by using a new approach relying on a kind of comparison principle for second order shape derivatives. 
\end{itemize}
Precise statements of these results are given in Section \ref{sec:mainRes}.

\subsection{Mathematical setup}
Throughout this article, $m_0$, $\kappa$ are fixed positive parameters. 
Since  in our work we want to extend the  results of \cite{LamboleyLaurainNadinPrivat}, let us define the set of admissible functions 
$$
\mathcal M_{m_0,\kappa}(\Omega)=\left\{m\in L^\infty(\O)\, , 0\leq m\leq \kappa\, ,\fint_\O m=m_0\right\},
$$
where $\fint_\O m$ denotes the average value of $m$ (see Section~\ref{sec:notations}) and assume that $m_0<\kappa$ so that $\mathcal M_{m_0,\kappa}(\O)$ is non-empty.
Given $\alpha\geq 0$ and $m\in \mathcal M_{m_0,\kappa}(\O)$, the operator $\mathcal L_m^\alpha$ is symmetric and compact. According to the spectral theorem, it is diagonalizable in $L^2(\O)$. In what follows, let  $\lambda_\alpha(m)$ be the first eigenvalue for this problem. According to the Krein-Rutman theorem, $\lambda_\alpha(m)$ is simple and its associated $L^2(\O)$-normalized eigenfunction $u_{\alpha,m}$ has a constant sign, say $u_{\alpha,m}\geq 0$. Let $R_{\alpha,m}$ be the associated Rayleigh quotient given by
\begin{equation}
R_{\alpha,m}:W^{1,2}_0(\O)\ni u\mapsto \frac{\int_\O (1+\alpha m)|\n u|^2-\int_\O mu^2}{\int_\O u^2}.\end{equation}
We recall that $\lambda_\alpha(m)$ can also be defined through the variational formulation 
\begin{equation}\label{Eq:Rayleigh} \lambda_{\alpha}(m):=\inf_{u\in W^{1,2}_0(\O)\, , u\neq 0} R_{\alpha,m}(u)=R_{\alpha,m}(u_{\alpha,m}).
\end{equation}
and that $u_{\alpha,m}$ solves 
\begin{equation}\label{Eq:EigenFunction}
\left\{
\begin{array}{ll}
-\n \cdot \Big((1+\alpha m)\n u_{\alpha,m}\Big)-mu_{\alpha,m}=\lambda_\alpha(m)u_{\alpha,m}&\text{ in }\O,
\\u_{\alpha,m}=0\text{ on }\partial \O.
\end{array}
\right.
\end{equation}
in a weak $W^{1,2}_0(\O)$ sense.
In this article, we  address the optimization problem
\begin{equation}\label{Pb:OptimizationEigenvalue}\tag{$P_\alpha$}\fbox{$\displaystyle
\inf_{m\in \mathcal M_{m_0,\kappa}(\O)}\lambda_\alpha(m).$}
\end{equation}
This problem is a modified version of the standard two-phase problem studied in \cite{MuratTartar,CoxLipton}; we detail the bibliography associated with this problem in Subsection \ref{SSE:BackgroundH}. It is notable that it is relevant in the framework of population dynamics,
when looking for optimal resources configurations in a heterogeneous environment for species survival, see Section \ref{BiologicalMotivations}.

\subsection{Main results}\label{sec:mainRes}
Before providing the main results of this article, we state a first fundamental property of the investigated model, reducing in some sense the research of general minimizers to the one of {\it bang-bang} densities. It is notable that, although the set of {\it bang-bang} densities is known to be dense in the set of all densities for the weak-star topology, such a result is not obvious since it rests upon continuity properties of $\lambda_\alpha$ for this topology. We overcome this difficulty by exploiting a convexity-like property of $\lambda_\alpha$.

\begin{proposition}[weak {\it bang-bang} property]\label{Th:BangBang}
Let $\O$ be a bounded connected subset of $\R^n$ with a Lipschitz boundary and let $\alpha>0$ be given. For every $m\in \mathcal M_{m_0,\kappa}(\O)$, there exists a {\it bang-bang} function $\tilde m\in \mathcal M_{m_0,\kappa}(\O)$ such that 
$$\lambda_\alpha(m)\geq \lambda_\alpha(\tilde m).$$ 
Moreover, if $m$ is not {\it bang-bang}, then we can choose $\tilde m$ so that the previous inequality is strict. 
\end{proposition}
In other words, given any resources distribution $m$, it is always possible to construct a {\it bang-bang} function $\tilde m$ that improves the criterion.

\paragraph{Non-existence for general domains.}
In a series  of paper, \cite{CasadoDiaz1,CasadoDiaz2,CasadoDiaz3}, Casado-Diaz proved that the problem of minimizing the first eigenvalue of the operator $u\mapsto -\n \cdot(1+\alpha m)\n u$ with respect to $m$ does not have a solution when $\partial \O$ is connected. His proof relies on a study of the regularity for this minimization problem, on homogenization and on a Serrin type argument. The following result is in the same vein, with two differences: it is weaker than his in the sense that it needs to assume higher regularity of the optimal set, but stronger in the sense that we do not make any strong assumption on $\partial \O$.  For further details regarding this literature, we refer to Section \ref{TwoPhase}.

\begin{theorem}\label{Th:NonExistence}
Let $\O$ be a bounded connected subset of $\R^n$ with a Lipschitz connected boundary, let $\alpha>0$ and $n\geq 2$. If the  optimization problem \eqref{Pb:OptimizationEigenvalue}  has a solution $\hat m \in \mathcal M_{m_0,\kappa}(\O)$, then this solution writes $\hat m=\kappa \mathds{1}_{\hat E}$, where $\hat E$ is a  measurable subset  of $\O$. Moreover, if $\partial \hat E$ is a $\mathscr C^2$ hypersurface and if $\O$ is connected,  then $\O$ is a ball.
\end{theorem}
The proof of this Theorem relies on methods developed by  Murat and Tartar, \cite{MuratTartar},  Cox and Lipton, \cite{CoxLipton}, and on a Theorem by Serrin \cite{Serrin1971}.


\paragraph{Analysis of optimal configurations in a ball.}
According to Theorem \ref{Th:NonExistence}, existence of regular solutions fail when $\O$ is not a ball. This suggest to investigate the case $\O=\B(0,R)$, which is the main goal of what follows.

Let us stress that proving the existence of a minimizer in this setting and characterizing it is a hard task. Indeed, to underline the difficulty, notice in particular that none of the usual rearrangement techniques (the Schwarz rearrangement or the Alvino-Trombetti one, see Section~\ref{TwoPhase}), that enable in general to reduce the research of solutions to radially symmetric densities, and thus to get compactness properties, can be applied here.

\begin{center}
\textsf{\large{The case of radially symmetric distributions}}
\end{center}
Here, we assume that $\O$ denotes the ball $\mathbb B(0,R)$ with $R>0$.  We define $r_0^*>0$ as the unique positive real number such that 
$$\kappa \frac{\left|\mathbb B(0,r_0^*)\right|}{\left|\mathbb B(0,R)\right|}=V_0.
$$
Let 
{\begin{equation} \label{eq:m0*}m_0^*=\kappa \mathds{1}_{\mathbb B(0,r_0^*)}=\kappa \mathds{1}_{E^*_0}\end{equation} }
be the centered distribution known to be the unique minimizer of $\lambda_0$ in $\mathcal M_{m_0,\kappa}(\O)$ (see e.g. \cite{LamboleyLaurainNadinPrivat}). 

In what follows, we restrict ourselves to the case of radially symmetric resources distributions.

\begin{theorem}\label{Th:RadialStability}
Let $\O=\mathbb B(0,R)$ and let $\mathcal M_{rad}$ be the subset of radially symmetric distributions of $ \mathcal M_{m_0,\kappa}(\O)$. The optimization problem
$$\inf_{m\in \mathcal M_{rad}}\lambda_\alpha(m)$$ has a solution. Furthermore, when $n=1,2,3$, there exists $\alpha^*>0$ such that, for any $\alpha<\alpha^*$, there holds
\begin{equation}
\min_{m\in \mathcal M_{rad}}\lambda_\alpha(m)=\lambda_\alpha(m^*_0).\end{equation}\end{theorem}
The proof of the existence part of the theorem relies on rearrangement techniques that were first introduced by Alvino and Trombetti in \cite{AlvinoTrombetti} and then refined in \cite{ConcaMahadevanSanz}.   
The stationarity result, i.e the fact that $m_0^*$ is a minimizer among radially symmetric distributions, was proved in the one-dimensional case in \cite{CaubetDeheuvelsPrivat}. To extend this result to higher dimensions, we developed an approach involving a homogenized version of the problem under consideration. The small dimensions hypothesis is due to a technical reason, which arises when dealing with elliptic regularity for this equation. 

Restricting ourselves to radially symmetric distributions might appear surprising since one could expect this result to be true without restriction, in $ \mathcal M_{m_0,\kappa}(\O)$. For instance, a similar result has been shown in the framework of two-phase eigenvalues  \cite{ConcaMahadevanSanz}, as a consequence of the Alvino-Trombetti rearrangement.  Unfortunately, regarding Problem \eqref{Pb:OptimizationEigenvalue},  no standard rearrangement technique leads to the { conclusion}, because of the specific form of the involved Rayleigh quotient. A first attempt in the investigation of the ball case is then to consider the case of radially symmetric distributions. It is notable that, even in this case, the proof appears unexpectedly difficult.

Finally,  we note that, as a consequence of the methods developed to prove Theorem~\ref{Th:RadialStability}, when a small amount of resources is available, the centered distribution $m^*_0$ is optimal among all resources distributions, regardless of radial symmetry assumptions.

\begin{corollary}\label{Th:Sketch}
Let $\O=\mathbb B(0,R)$ and $m_0^*$ be defined {by (\ref{eq:m0*})} There exist $\underline m>0$, $\underline \alpha>0$ such that, if 
$m_0\leq \underline m$ and $\alpha<\underline\alpha$, then the unique solution of \eqref{Pb:OptimizationEigenvalue} is $m^*_0=\kappa \mathds 1_{E^*_0}.$
\end{corollary}

\begin{center}
\textsf{\large{Local stability of the ball distribution with respect to Hadamard perturbations of resources sets}}
\end{center}
In what follows, we tackle the issue of  the local minimality  of $m^*_0$ in $\mathcal M_{m_0,\kappa}(\O)$ with the help of a shape derivative approach. We obtain partial results in dimension $n=2$.

Let $\O$ be a bounded connected domain with a Lipschitz boundary, and consider a {\it bang-bang} function $m\in \mathcal M_{m_0,\kappa}(\O)$ writing $m=\kappa \mathds{1}_E$, for a measurable subset $E$ of $\O$ such that 
$\kappa |E|=m_0|\O|$. Let us write $\lambda_\alpha(E):=\lambda_\alpha\left(\mathds{1}_E\right)$, with a slight abuse of notation.
Let us assume that $E$ has a $\mathscr C^2$ boundary. Let $V:\O\rightarrow \R^n$ be a $W^{3,\infty}$ vector field with compact support, and define for every $t$ small enough, $E_t:=\left(\operatorname{Id}+tV\right)E$.
For $t$ small enough, $\phi_t:=\operatorname{Id}+tV$ is a smooth diffeomorphism from $E$ to $E_t$, and $E_t$ is an open connected set with a $\mathscr C^2$ boundary.
 If $\mathcal F:E\mapsto \mathcal F(E)$ denotes a shape functional, the first (resp. second) order shape derivative of $\mathcal F$ at $E$ in the direction $V$ is 
$$
\mathcal F'(E)[V]:=\left.\frac{d}{dt}\right|_{t=0}\mathcal F(E_t) \quad \left(\text{ resp. }\left.\frac{d^2}{d t^2}\right|_{t=0} \mathcal F(E_t)\right)
$$
whenever these quantities exist.

For further details regarding the notion of shape derivative, we refer to \cite[Chapter 5]{HenrotPierre}.

It is standard to write optimality conditions in terms of a sort of tangent space for the measure constraint: indeed, since the volume constraint $\operatorname{Vol}(E_t)=m_0\operatorname{Vol}(\O)/\kappa$ is imposed, we will deal with vector fields $V$ satisfying the linearized volume condition $\int_E \n \cdot V=0$. We thus call \textit{admissible at $E$} such vector fields and introduce
\begin{equation}\label{Eq:X}\mathcal X(E):=\left\{V\in W^{3,\infty}(\R^n;\R^n)\, , \int_E \n \cdot V=0\, , \Vert V\Vert _{W^{3,\infty}}\leq 1\right\}.
\end{equation}

A shape $E\subset \O$ with a $\mathscr C^2$ boundary such that $\kappa |E|=m_0|\O|$ is said to be critical if
\begin{equation}\label{Eq:FOO}
\forall V \in \mathcal X(E),\quad \lambda_\alpha'(E)[V]=0.
\end{equation}
or equivalently, if there exist a Lagrange multiplier $\Lambda_\alpha$ such that $ \left(\lambda_\alpha'-\Lambda_\alpha \operatorname{Vol}'\right)(E)[V]= 0$ for all $V \in W^{1,\infty}(\O)$, where $\operatorname{Vol}:\O\mapsto |\O|$ denotes the volume functional. 
Furthermore, if $E$ is a local minimizer for Problem \eqref{Pb:OptimizationEigenvalue}, then one has
\begin{equation}\label{Eq:SOO}
\forall V \in \mathcal X(E), \quad \left(\lambda_\alpha''-\Lambda_\alpha \operatorname{Vol}''\right)(E)[V,V]\geq 0.
\end{equation}
In what follows, we will still assume that $\O$ denotes the ball $\mathbb B(0,R)$ with $R>0$.
\begin{theorem}\label{Th:ShapeStability}
Let us assume that $n=2$ and that $\O=\mathbb B(0,R)$. The ball $E=\mathbb B(0,r_0^*)=\mathbb B^*$ satisfies the shape optimality conditions \eqref{Eq:FOO}-\eqref{Eq:SOO}. Furthermore, if $\Lambda_\alpha$ is the Lagrange multiplier associated with the volume constraint, there exist two constants $\overline\alpha>0$ and $C>0$ such that, for any $\alpha\in [0,\overline \alpha)$ and any vector field $V\in \mathcal X(\B^*)$, there holds
$$\left(\lambda_\alpha''-\Lambda_\alpha \operatorname{Vol}''\right)(\B^*)[V,V]\geq C\Vert  V\cdot \nu  \Vert _{L^2(\mathbb S^*)}^2.$$
\end{theorem}
\begin{remark}The proof requires explicit  computation of the shape derivative of the eigenfunction. We note that in \cite{DambrineKateb} such computations are carried out for the two-phase problem and that in \cite{KaoLouYanagida} such an approach is undertaken to investigate the stability of certain configurations for a weighted Neumann eigenvalue problem.

The main contribution of this result is to shed light on a monotonicity principle that enables one 
to lead a careful asymptotic analysis of the second order shape derivative of the functional as $\alpha \to 0$. It is important to note that, although this allows us to deeply analyze the second order optimality conditions, it is expected that the optimal coercivity norm in the right-hand side above is expected to be $H^{\frac12}$ whenever $\alpha>0$, which we do not recover with our method. The reason why we believe that in this context the optimal coercivity norm is $H^{\frac12}$ is that in \cite{DambrineKateb}, precise computations for the two-phase problem \eqref{Pg:TwoPhase} are carried out and a $H^{\frac12}$ coercivity norm is obtained for certain classes of parameters. On the other hand, when $\alpha =0$, it was shown in \cite{MazariQuantitative} that the optimal coercivity norm is $L^2$, and a quantitative inequality was then derived. 
\end{remark}

\begin{remark}
We believe that our strategy of proof may be used to obtain the same kind of coercivity norm in the three-dimensional case. However,  we believe that such a generalization would be non-trivial and need technicalities. Since the main contribution is to introduce a methodology to study the positivity of second-order shape derivative, we simply provide a possible strategy to prove the result in the three dimensional case in the concluding section of the proof of Theorem~\ref{Th:ShapeStability}, see Section~\ref{Se:CclSha}.
\end{remark}

The rest of this article is dedicated to proofs of the results we have just outlined.

\subsection{A biological application of the problem}\label{BiologicalMotivations}
Equation \eqref{Eq:EigenFunction} arises naturally when dealing with simple population dynamics in heterogeneous spaces. 

Let $\e\geq 0$ be a parameter of the model. We consider a population density whose flux is given by 
$$
\mathcal J_\e=-\n u+\e u \n m.
$$ 
Since $\n m$ might not make sense if $m$ is assumed to be only measurable, we temporarily omit this difficulty by assuming it smooth enough so that the expression above makes sense. The term $u \n m$ appears as a drift term and stands for a bias in the  population movement, modeling a tendency of the population to disperse along the gradient of resources and hence move to favorable regions. The parameter $\e$ quantifies the influence of the resources distribution on the movement of the species. 
The complete associated reaction diffusion equation, called ``logistic diffusive equation'', reads
$$
\frac{\partial u}{\partial t}=\n \cdot \Big(\n u-\e u \n m\Big)+mu-u^2\quad \text{in }\O,
$$
completed with suitable boundary conditions. In what follows, we will focus on {Dirichlet boundary conditions} meaning that the boundary of $\O$ is lethal for the population living inside. 
Plugging the change of variable $v=e^{-\e m}u$  in this equation leads to 
$$
\frac{\partial v}{\partial t}=\Delta v+\e \n m\cdot \n u +mv-e^{\e m}v\quad \text{in }\O.
$$ 
It is known (see e.g. \cite{BelgacemCosner,Belgacem,Murray}) that the asymptotic behavior of this equation is driven by
the principal eigenvalue of the operator $\tilde {\mathcal L}:u\mapsto -\Delta u-\e  \n m\cdot \n u -mu$.
The associated principal eigenfunction $\psi$ satisfies
$$
-\n \cdot \left(e^{\e m}\n \psi\right)-me^{\e m}\psi=\tilde \lambda_\e \psi e^{\e m}\quad \text{in }\O.
$$
Following the approach developed in \cite{LamboleyLaurainNadinPrivat}, optimal configurations of resources correspond to the ones ensuring the fastest convergence to the steady-states of the PDE above, which comes to
minimizing $\tilde \lambda_\e(m)$ with respect to $m$.

By using Proposition \ref{Th:BangBang}, which enables us to only deal with {\it bang-bang} densities $m$, one shows easily that minimizing $\tilde \lambda_\e(m)$ over $\mathcal M_{m_0,\kappa}(\O)$  is equivalent to minimizing $\lambda_\e(m)$ over $\mathcal M_{m_0,\kappa}(\O)$, in other words to
Problem \eqref{Pb:OptimizationEigenvalue} with $\alpha=\e$. Theorem \ref{Th:NonExistence} can thus be interpreted as follows in this framework: assuming that the population density moves along the gradient of the resources, it is  not possible to lay the resources in an optimal way. 
Note that the conclusion is completely different in the case $\alpha=0$ (see \cite{LamboleyLaurainNadinPrivat}) or in the one-dimensional case (i.e.  $\O=(0,1)$) with $\alpha >0$ (see \cite{CaubetDeheuvelsPrivat}), where minimizers exist. 
In the last case, optimal configurations for three kinds boundary conditions (Dirichlet, Neumann, Robin) have been obtained, by using a new rearrangement technique. 
Finally, let us mention the related result \cite[Theorem 2.1]{Hamel2011},  dealing with Faber-Krahn type inequalities for general operators of the form
$$
\mathcal K:u\mapsto -\n \cdot(A\n  u)- V\cdot \n u-mu
$$ 
where $A$ is a positive symmetric matrix.
Let us denote the first eigenvalue of $\mathcal K$ by $E(A,V,m)$. It is shown, by using new rearrangements, that there exist radially symmetric elements $A^*,V^*,m^*$ such that
$$
0< \inf A\leq A^*\leq \Vert A\Vert _\infty,\quad \Vert A^{-1}\Vert _{L^1}=\Vert (A^*)^{-1}\Vert _{L^1},\quad
\Vert V^*\Vert _{L^\infty}\leq \Vert V\Vert _{L^\infty}
$$
and $E(A,V,m)\geq E(A^*,V^*,m^*)$.
We note that applying this result directly to our problem would not allow us to conclude. Indeed, we would get that for every $\O$ of volume $V_1$ and every $m\in \mathcal M_{m_0,\kappa}(\O)$, if $\O^*$ is the ball of volume $V_1$, there exist two radially symmetric functions $m_1$ and $m_2$ satisfying $m_1$, $m_2$ in $\mathcal M_{m_0,\kappa}(\O)$ such that $\lambda_\alpha(m)\geq \mu_\alpha(m_1,m_2)$, where $\mu_\alpha(m_1,m_2)$ is the first eigenvalue of the operator $-\n \cdot((1+\alpha m_1)\n )-m_2$.
We note that this result could also be obtained by using the symmetrization techniques of \cite{AlvinoTrombetti}. 

Finally, let us mention optimal control problems involving a similar model but a different cost functional, related to:
\begin{itemize} 
\item the total size of the population for a logistic diffusive equation in \cite{Mazari2020,MNPChapter,Mazari2021}. 
\item optimal harvesting of a marine resource, investigated in the series of articles \cite{MR3035462,MR2476432,MR3628307}. 
\end{itemize} 

\subsection{Notations and notational conventions, technical properties of the eigenfunctions}\label{sec:notations}
Let us {sum up} the notations used throughout this article.
\begin{itemize}[label=\textbullet]
\item $\R_+$ is the set of non-negative real numbers. $\R_+^*$ is the set of positive real numbers.
\item  $n$ is a fixed positive integer and $\O$ is a bounded connected domain in $\R^n$.
\item if $E$ denotes a subset of $\O$, the notation $\mathds{1}_E$ stands for the characteristic function of $E$, equal to 1 in $E$ and 0 elsewhere. 
\item the notation $\Vert \cdot\Vert $ used without subscript refers to the standard Euclidean norm in $\R^n$. When referring to the norm of a Banach space $\mathcal X$, we write it $\Vert \cdot\Vert _{\mathcal X}$.
\item The average of every $f\in L^1(\O)$ is denoted by $\fint_\O f:=\frac1{|\O|}\int_\O f$.
\item $\nu$ stands for the outward unit normal vector on $\partial \O$.
\item if $m(\cdot)$ is a given function in $L^\infty(\O)$ and $\alpha$ a positive real number, we will use the notation $\sigma_{\alpha,m}$ to denote the function $1+\alpha m$. {When there is no ambiguity, we sometimes use the notation $\sigma_{\alpha}$ to alleviate notations.}
\item If $E$ denotes a subset of $\O$ with $\mathscr{C}^2$ boundary, we will use the notations 
$$
f|_{int}(y)=\lim_{x\in E,x\to y}f(x)\quad \text{and}\quad f|_{ext}(y):=\lim_{x \in (\O\backslash E),x\to y}f(y)
$$
so that $\llbracket f\rrbracket =f|_{ext}-f|_{int}$ denotes the jump of $f$ at $x\in \partial E$.
\end{itemize}


\section{Preliminaries}\label{Pr:BangBang}

\def\dlamh{{\dot \lambda_\alpha(m)[h]}}
\def\ddlamh{{\ddot \lambda_\alpha(m)[h]}}
\def\dduamh{{\ddot u_{\alpha,m}[h]}}

\subsection{Switching function}
To derive optimality conditions for Problem~\eqref{Pb:OptimizationEigenvalue}, we introduce the tangent cone to $\mathcal{M}_{m_0,\kappa}(\O)$ at any point of this set.

\begin{definition} (\cite[chapter 7]{HenrotPierre})\label{def:tgtcone}
For every $m\in \mathcal M_{m_0,\kappa}(\O)$, the tangent cone to the set $\mathcal{M}_{m_0,\kappa}(\O)$ at $m$, also called the \textit{admissible cone} to the set $\mathcal{M}_{m_0,\kappa}(\O)$ at $m$, denoted by $\mathcal{T}_{m}$ is the set of functions $h\in L^\infty(\O)$ such that, for any sequence of positive real numbers $\varepsilon_n$ decreasing to $0$, there exists a sequence of functions $h_n\in L^\infty(\O)$ converging to $h$ as $n\rightarrow +\infty$, and $m+\varepsilon_nh_n\in\mathcal{M}_{m_0,\kappa}(\O)$ for every $n\in\N$.\label{footnote:cone}
\end{definition}

Notice that, as a consequence of this definition, any $h \in \mathcal T_m$ satisfies $\fint_\O h=0$. 
\begin{lemma}\label{diff:valPvecP}
Let $m\in  \mathcal M_{m_0,\kappa}(\O)$ and $h\in  \mathcal T_m$. The mapping $ \mathcal M_{m_0,\kappa}(\O)\ni m\mapsto u_{\alpha,m}$ is twice differentiable at $m$ in direction $h$ in a strong $L^2(\O)$ sense and in a weak $W^{1,2}_0(\O)$ sense, and the mapping $ \mathcal M_{m_0,\kappa}(\O)\ni m\mapsto \lambda_\alpha$ is twice differentiable in a strong $L^2(\O)$ sense.
\end{lemma}
The proof of this lemma is technical and  is postponed to Appendix \ref{Ap:Differentiability}.

For $t$ small enough, let us introduce the mapping $g_h:t\mapsto \lambda_\alpha \left([m+th]\right)$. Hence, $g_h$ is twice differentiable. 
The first and second order derivatives of $\lambda_\alpha$ at $m$ in direction $h$, denoted by $\dot\lambda_\alpha(m)[h]$ and $\ddot\lambda_\alpha(m)[h]$, are defined by
$$
\dot \lambda_\alpha(m)[h]:=g_h'(0)\quad\text{and}\quad \ddot\lambda_\alpha(m)[h]:=g_h''(0).
$$ 
\begin{lemma}\label{Le:DeriveeL21}
Let $m\in \mathcal{M}_{m_0,\kappa}(\O)$ and $h\in \mathcal{T}_m$. The mapping $m\mapsto \lambda_\alpha(m)$ is differentiable at $m$ in direction $h$ in $L^2$ and its differential reads
\begin{equation}\label{Eq:Derivee21}
\dot \lambda_\alpha(m)[h]=\int_\O h\psi_{\alpha,m},\quad \text{with}\quad \psi_{\alpha,m}:=\alpha  |\n u_{\alpha,m}|^2-u_{\alpha,m}^2.
\end{equation}
The function $\psi_{\alpha,m}$ is called {\it switching function}.
\end{lemma}
\begin{proof}
According to Lemma~\ref{diff:valPvecP}, we can differentiate the variational formulation associated to \eqref{Eq:EigenFunction} and get that the differential $\duamh$ of  $m\mapsto u_{\alpha,m}$ at $m$ in direction $h$  satisfies, with $\sigma_\alpha :=1+\alpha m$,
\begin{equation}\label{Eq:L2Derivative1}\left\{\begin{array}{lll}
-\n \cdot\Big(\sigma_\alpha\n \duamh\Big)-\alpha \n \cdot\Big(h\n \uam\Big)=&\dlamh \uam +\lambda_\alpha(m)\duamh & \\
& +m\duamh+h\uam  & \text{ in }\O,\\
\duamh=0  {\text{ on } \partial \O},\\
\int_\O \uam \duamh=0.&
\end{array}\right.
\end{equation}
Multiplying {(\ref{Eq:L2Derivative1})} by $\uam$, integrating by parts and using that $\uam$ is normalized in $L^2(\O)$ leads to 
\begin{align*}
\dlamh =&\underbrace{\int_\O \sigma_\alpha \n \duamh\cdot \n \uam -\lambda_{\alpha}(m)\int_\O \uam \duamh-\int_\O m\uam \duamh }_{=0\text{ according to \eqref{Eq:EigenFunction}}}\\
&+\int_\O \alpha h |\n \uam|^2-\int_\O h u_{\alpha,m}^2.
\end{align*}
\end{proof}

\subsection{Proof of Proposition~\ref{Th:BangBang}}
The proof relies on concavity properties of the functional $\lambda_\alpha$. More precisely, let  $m_1,m_2\in \mathcal M_{m_0,\kappa}(\O)$. We will show that the map $f:[0,1]\ni t\mapsto \lambda_\alpha \left((1-t)m_1+tm_2\right)$ is strictly concave, i.e that $ f''< 0$ on $[0,1]$.

Note that the characterization of the concavity in terms of second order derivatives makes sense, according to Lemma~\ref{diff:valPvecP}, since $\lambda_\alpha$ is twice differentiable. Before showing this concavity property, let us first explain why it implies the conclusion of Proposition~\ref{Th:BangBang} (the weak {\it bang-bang} property).
{Suppose that $m\in \mathcal M_{m_0,\kappa}(\O)$ is not { bang-bang}}. The set $\mathcal{I}=\{0<m<\kappa\}$ is then of positive Lebesgue measure and $m$ is therefore not extremal in $ \mathcal M_{m_0,\kappa}(\O)$, according to \cite[Prop.~7.2.14]{HenrotPierre}.
We then infer the existence of $t \in (0,1)$ as well as two distinct elements $m_1$ and $m_2$ of $ \mathcal M_{m_0,\kappa}(\O)$ such that $m=(1-t) m_1+tm_2$. 
Because of the strict concavity of $\lambda_\alpha$, the solution of the optimization problem $\min \{\lambda_\alpha((1-t) m_1+tm_2)\}$ is either $m_1$ or $m_2$, and moreover, $m$ cannot solve this problem. 
Assume that $m_1$ solves this problem without loss of generality. One thus has $\lambda_\alpha(m_1)<\lambda_\alpha (m)$. Since the subset of {\it bang-bang} functions of $\mathcal M_{m_0,\kappa}(\O)$ is dense in $\mathcal M_{m_0,\kappa}(\O)$ for the weak-star topology of $L^\infty(\O)$, there exists a sequence of {\it bang-bang} functions $(m^k)_{k\in \N}$ of $\mathcal M_{m_0,\kappa}(\O)$ converging weakly-star to $m_1$ in $L^\infty(\O)$. Furthermore, $\lambda_\alpha$ is upper semicontinuous for the weak-star topology of $L^\infty(\O)$, since it reads as the infimum of continuous linear functionals for this topology. Let $\varepsilon>0$. We infer the existence of $k_\varepsilon\in \N$ such that $\lambda_\alpha(m^{k_\varepsilon})\leq \lambda_\alpha(m_1)+\varepsilon$. By choosing $\varepsilon$ small enough, we get that $\lambda_\alpha(m^{k_\varepsilon})<\lambda_\alpha (m)$, whence the result.

\medskip

It now remains to prove that $f$ is strictly concave. Let $m\in \mathcal M_{m_0,\kappa}(\O)$, and set $m_1=m$, $h=m_2-m_1$, we observe that $f''(t)=\ddot\lambda_{\alpha}((1-t)m_1+tm_2)[h]$ for all $t\in [0,1]$.
The differential $\duamh$ of  $m\mapsto u_{\alpha,m}$ at $m$ in direction $h$, denoted $\duamh$, satisfies \eqref{Eq:L2Derivative1} and the second order {Gateaux derivatives} $\dduamh$ and $\ddlamh$ solve, with $\sigma_\alpha:=1+\alpha m$,
\begin{equation}\label{Eq:L2Derivative2}
\left\{\begin{array}{lll}
-\n \cdot\Big(\sigma_\alpha\n \dduamh\Big)-2\alpha \n \cdot\Big(h\n \duamh\Big)=&\ddlamh \uam +2\dlamh \duamh \\
&+\lambda_\alpha(m)\dduamh+m\dduamh+2h\duamh & \text{ in }\O,\\
\dduamh=0 \quad {\text{ on } \partial \O.}&&
\end{array}\right.\end{equation}
Multiplying {(\ref{Eq:L2Derivative2})} by $\uam$, using that $\uam$ is normalized in $L^2(\O)$ and integrating by parts yields
\begin{align*}
\ddlamh&=\underbrace{\int_\O \sigma_\alpha \n \dduamh\cdot \n \uam-\lambda_{\alpha}(m)\int_\O \uam \dduamh-\int_\O m\uam \dduamh }_{=0\text{ according to \eqref{Eq:EigenFunction}}}
\\&+2\alpha \int_\O h \n \duamh,\n \uam-2\int_\O h\uam\duamh
\\&=2\left(-\int_\O \sigma_\alpha |\n \duamh|^2+\int_\O m\duamh^2+\lambda_\alpha(m)\int_\O \duamh^2\right)+2\underbrace{\dlamh \int_\O u_{\alpha,m} \duamh}_{=0\text{ since $\int_\O \uam \duamh=0$}}
\\&=2\int_\O \duamh^2\left(-R_{\alpha,m}[\duamh]+\lambda_\alpha(m)\right)< 0,
\end{align*}
where the last inequality comes from the observation that, whenever $h\neq 0$, one has $\duamh\neq 0$ and $\duamh$ is in the orthogonal space to the first eigenfunction $\uam$ in $L^2(\O)$. Since the first eigenvalue is simple, the Rayleigh quotient of $\duamh$ is greater than $\lambda_{\alpha}(m)$.


\section{Proof of Theorem \ref{Th:NonExistence}}\label{Pr:NonExistence}
This proof is based on a homogenization argument, inspired from the notions and techniques introduced in \cite{MuratTartar}. In the next section, we gather the preliminary tools and material involved in what follows.\subsection{Background material on homogenization and bibliographical comments}\label{SSE:BackgroundH}
Let us recall several usual definitions and results in homogenization theory we will need hereafter.
\begin{definition}[$H$-convergence]
Let $(m_k)_{k\in \N}\in \mathcal M_{m_0,\kappa}(\O)^\N$ and for every $k\in \N$, define respectively $\sigma_k$ and $u_k(f)$ by $\sigma_k=1+\alpha m_k$ and as the unique solution of
$$
\left\{\begin{array}{ll}-\n \cdot(\sigma_k \n u_k(f))=f&\text{ in }\O\\
u_k(f)=0 \text{ on }\partial \O
\end{array}
\right.
$$
where $f\in L^2(\O)$ is given. We say that the sequence $(\sigma_k)_{k\in \N}$  $H$-converges to $A:\O\to M_{n}(\R)$ if, for every $f\in L^2(\O)$, the sequence $(u_k(f))_{k\in \N}$ converges weakly to $u_\infty$ in $\wo$ and the sequence $(\sigma_k\n u_k)_{k\in \N}$ converges weakly to $A\n u_\infty $ in $ L^2(\O)$, where $u_\infty$ solves
$$\left\{\begin{array}{ll}-\n \cdot(A\n u_\infty)=f&\text{ in }\O\, , 
\\u_\infty=0 \text{ on }\partial \O\end{array}
\right.$$
In that case, we will write $\sigma_k \overset{H}{\underset{k\to \infty}\longrightarrow} A$.
\end{definition}
\begin{definition}[arithmetic and geometric means]\label{De:Homo}
Let $m\in \mathcal M_{m_0,\kappa}(\O)$ and $\sigma=1+\alpha m$. We define the arithmetic mean of $\sigma$ by $\Lambda_+(m)=\sigma$, and its harmonic mean by $\Lambda_-(m)=\frac{1+\alpha \kappa}{1+\alpha(\kappa-m)}$.
One has $\Lambda_-(m)\leq \Lambda_+(m)$, according to the arithmetic-harmonic inequality, with equality if  and only if
$m$ is a {\it bang-bang} function.
\end{definition}
\begin{prnonumbering}\cite[Proposition 10]{MuratTartar}
Let $(m_k)_{k\in \N}\in \mathcal M_{m_0,\kappa}(\O)^\N$ and $(\sigma_k)_{k\in \N}$ given by $\sigma_k=1+\alpha m_k$. Up to a subsequence, there exists $m\in \mathcal M_{m_0,\kappa}(\O)$ such that $(m_k)_{k\in \N}\in \mathcal M_{m_0,\kappa}(\O)^\N$ converges to $m$ for the weak-star topology of $L^\infty$.

Assume moreover that the sequence $(\sigma_k)_{k\in \N}$ $H$-converges to a matrix $A$. Then, $A$ is a symmetric matrix, its spectrum $\Sigma(A)=\{\lambda_1,\dots,\lambda_n\}$ is real, and
\begin{equation} \label{Spec}\tag{$J_1$} \Lambda_-(m)\leq \min \Sigma(A)\leq \max \Sigma(A)\leq \Lambda_+(m).\end{equation}
\begin{equation}\label{Spec1}\tag{$J_2$}\sum_{j=1}^n \frac1{\lambda_j-1}\leq \frac1{\Lambda_-(m)-1}+\frac{n-1}{\Lambda_+(m)-1},\end{equation}
\begin{equation} \label{Spec2}\tag{$J_3$}\sum_{j=1}^n \frac1{1+\alpha\kappa-\lambda_j}\leq \frac1{1+\alpha \kappa-\Lambda_-(m)}+\frac{n-1}{1+\alpha\kappa -\Lambda_+(m)}.\end{equation}
\end{prnonumbering}
For a given $m\in \mathcal M_{m_0,\kappa}(\O)$, we introduce 
$$
M_m^\alpha=\{A:\O\rightarrow S_n(\R)\, , A \text{ satisfies \eqref{Spec}-\eqref{Spec1}-\eqref{Spec2}}\}.
$$
For a matrix-valued application $A\in M_m^\alpha$ for some $m\in \mathcal M_{m_0,\kappa}(\O)$, it is possible to define the principal eigenvalue of $A$ via Rayleigh quotients as 
\begin{equation}\label{Eq:Zeta}
\zeta_\alpha(m,A):=\inf_{u\in \wo\, , \int_\O u^2=1}\int_\O  A\n u\cdot \n u-\int_\O mu^2.
\end{equation}
Note that the dependence of $\zeta_\alpha$ on the parameter $\alpha$ is implicitly contained in the condition $A\in M_m^\alpha$. We henceforth focus on the following relaxed version of the optimization problem:
\begin{equation}\label{Pb:OptimizationRelax}
\inf_{m\in \mathcal M_{m_0,\kappa}(\O)\, , A\in M_m^\alpha} \zeta_\alpha(m,A).
\end{equation}
for which we have the following result.
\begin{thmnonumbering}\cite[Proposition 10]{MuratTartar}
\begin{itemize}
\item[(i)] For every $m\in \mathcal M_{m_0,\kappa}(\O)$ and $A \in M_m^\alpha$, there exists a sequence $(m_k)_{k\in \N}\in \mathcal M_{m_0,\kappa}(\O)$ such that $(m_k)_{k\in \N}$ converges to $m$ for the weak-star topology of $L^\infty$, and the sequence $(\sigma_k)_{k\in \N}$ defined by $\sigma_k=1+\alpha m_k$ $H$-converges to $A$, as $k\to +\infty$. 
\item[(ii)] The mapping $(m,A)\mapsto \lambda_\alpha(m,A)$ is continuous with respect to the $H$-convergence (see in particular \cite{Oleinik}).
\item[(iii)] The variational problem  \eqref{Pb:OptimizationRelax}
has a solution $(\hat{m},\hat{A})$; by definition, $\hat{A}\in M_{\hat{m}}^\alpha$.  Furthermore, if $\hat{u}$ is the associated eigenfunction, then $\hat{A} \n \hat{u} =\Lambda_-(\hat{m})\n \hat{u}$.
\end{itemize}
\end{thmnonumbering}
This theorem allows us to solve Problem~\eqref{Pb:OptimizationRelax}.
\begin{corollary}\cite{MuratTartar}\label{Co:Equivalence}
If Problem~\eqref{Pb:OptimizationEigenvalue} has a solution $\hat{m}$, then the couple $(\hat{m},1+\alpha \hat{m})$ solves Problem~\eqref{Pb:OptimizationRelax}.
\end{corollary}
\begin{proof}[Proof of Corollary \ref{Co:Equivalence}]
Assume that the solution of \eqref{Pb:OptimizationRelax} is $(\hat{m}, \hat{A})$ and that $\hat{A}\neq 1+\alpha \hat{m}$. Then there exists a sequence $(m_k)_{k\in \N}$ converging
weak-star in $L^\infty$ to $\hat{m}$ and such that the sequence $(1+\alpha m_k)_{k\in \N}$ $H$-converges to $\hat{A}$. This means that 
$$\lambda_\alpha(\hat{m})=\zeta_\alpha(\hat{m},1+\alpha \hat{m})>\zeta_\alpha(\hat{m},\hat{A})=\underset{k\to \infty}\lim \lambda_\alpha(m_k)$$ which immediately yields a contradiction.
\end{proof}

Let us end this section with several bibliographical comments on such problems.

\paragraph{Bibliographical comments on the two-phase conductors problem.}\label{TwoPhase}

Problem \eqref{Pg:TwoPhase} {with $\mathcal M:=\mathcal M_{m_0,\kappa}(\Omega)$} has drawn a lot of attention in the last decades, since the seminal works by Murat and Tartar, \cite{Murat,MuratTartar}
Roughly speaking, this optimal design problem is, in general, ill-posed and one needs to introduce a relaxed formulation to get existence.  
We refer to \cite{Allaire,CoxLipton,MuratTartar,Oleinik}.

Let us provide the main lines strategy to investigate existence issues for Problem \eqref{Pg:TwoPhase}, according to \cite{Murat,MuratTartar}.
If the solution $(\hat{m},1+\alpha\hat{m})$ to the relaxed problem \eqref{Pb:OptimizationRelax}  is a solution to the original problem \eqref{Pg:TwoPhase}, then there exists a measurable subset $\hat{E}$ of $\O$ such that $\hat{m}=\kappa \mathds{1}_{\hat{E}}$. If furthermore $\hat{E}$ is assumed to be smooth enough, then, denoting by $\hat{u}$ the principal eigenfunction associated with $(\hat{m},\lambda_\alpha(m))=(\hat{m},\zeta_\alpha(\hat{m}, 1+\alpha \hat{m}))$, we get that $\hat{u}$ and $(1+\alpha \hat{m}) \frac{\partial \hat{u}}{\partial \nu}$ must be constant  on $\partial \hat{E}$.
The function $1+\alpha \hat{m}$ being discontinuous across $\partial E$, the optimality condition above has to be understood in the following sense: the function $(1+\alpha \hat{m}) \frac{\partial \hat{u}}{\partial \nu}$, a priori discontinuous, is in fact continuous across $\partial E$ and even constant on it.
Note that these arguments have been generalized in \cite{CoxLipton}. These optimality conditions, combined with Serrin's Theorem \cite{Serrin1971}, suggest that Problem \eqref{Pg:TwoPhase} could have a solution if, and only if $\O$ is a ball. The best results known to date are the following ones.
\begin{thmnonumbering}
\begin{itemize}
\item[(i)] Let $\O$ be an open set such that $\partial \O$ is $\mathscr C^2$ and connected. Problem \eqref{Pg:TwoPhase} has a solution if and only if $\O$ is a ball \cite{CasadoDiaz3}.
\item[(ii)] If $\O$ is a ball, then Problem \eqref{Pg:TwoPhase} has a solution which is moreover radially symmetric \cite{ConcaMahadevanSanz}.
\end{itemize}
\end{thmnonumbering}

Regarding the second part of the theorem, the authors used a particular rearrangement coming to replace $1+\alpha m$ by its harmonic mean on each level-set of the eigenfunction. Such a rearrangement has been first introduced by Alvino and Trombetti \cite{AlvinoTrombetti}. This drives the author to reduce the class of admissible functions to radially symmetric ones, which allow them to conclude thanks to a compactness argument \cite{AlvinoTrombettiLions}. These arguments are mimicked to derive the existence part of Theorem~\ref{Th:RadialStability}.

Finally, let us mention \cite{ConcaLaurainMahadevan,Laurain}, where the optimality of annular configurations in the ball is investigated.
A complete picture of the situation is then depicted in the case where $\alpha$ is small, which is often referred to as the "low contrast regime". 
We also mention \cite{DambrineKateb} , where a shape derivative approach is undertaken to characterize minimizers when $\O$ is a ball.

\subsection{Proof of Theorem \ref{Th:NonExistence}} 

Let us assume the existence of a solution to Problem \eqref{Pb:OptimizationEigenvalue}, denoted $\hat{m}$.  According to Proposition~\ref{Th:BangBang}, there exists a measurable subset $\hat{E}$ of $\O$ such that 
$\hat{m}=\kappa \mathds{1}_{\hat{E}}$. Let us introduce $\hat{\sigma}:=1+\alpha \hat{m}$ and $\hat{u}$, the $L^2$-normalized eigenfunction associated to $\hat{m}$.

Let us now assume that $\partial \hat{E}$ is $\mathscr C^2$.
\paragraph{Step 1: derivation of optimality conditions.}
What follows is an adaptation of \cite{CoxLipton}. For this reason, we only recall the main lines. 
Let us write the optimality condition 
 for the problem 
$$\min_{m\in \mathcal M_{m_0,\kappa}(\O)}\min_{A\in M_m^\alpha}\zeta_\alpha(m,A)=\lambda_\alpha(\hat{m}),$$
where $\zeta_\alpha$ is given by \eqref{Eq:Zeta}. Let $h$ be an admissible perturbation at $\hat{m}$. In \cite{MuratTartar} it is is proved that for every $\e>0$ small enough, there exists a matrix-valued application $A_\e \in M_{\hat{m}+\e h}$ such that
$$
A_\e\n \hat{u}=\Lambda_-(\hat{m}+\e h)\n \hat{u}\quad \text{in }\O,
$$
where $\Lambda_-$ has been introduced in Definition \ref{De:Homo}.
Fix $\e$ as above. Since $(\hat{m},1+\alpha \hat{m})$ is a solution of the Problem \eqref{Pb:OptimizationRelax}, one has
\begin{align*}
\int_\O  A_\e \n \hat{u}\cdot \n \hat{u}-\int_\O (m+\e h)\hat{u}^2&\geq \zeta_\alpha(m+\e h,A_\e) \\
&\geq \zeta_\alpha(\hat{m},1+\alpha \hat{m})=\int_\O \hat{\sigma} |\n \hat{u}|^2-\int_\O \hat{m}\, \hat{u}^2.
\end{align*}
where one used the Rayleigh quotient definition of $\zeta_\alpha$ as well as the minimality of $(\hat{m},1+\alpha \hat{m})$.
Dividing the last inequality by $\e$ and passing to the limit yields
$$\int_\O h\left.\frac{d\Lambda_-(m)}{dm}\right|_{m=\hat{m}}|\n \hat{u}|^2-h\hat{u}^2\geq 0.$$
Using that $d\Lambda_-/dm=\alpha \Lambda_-(m)^2/(1+\alpha\kappa)$, and that $\hat{m}$ is a {\it bang-bang} function (so that $\Lambda_-(\hat{m})=\hat{\sigma}$),
we infer that the first order optimality conditions read: there exists $\mu\in \R$ such that
\begin{equation}\label{Eq:OptimalityH} 
\left\{\Psi_\alpha< \mu\right\}\subset \hat{E}\subset \left\{\Psi_\alpha\leq \mu\right\}\quad \text{where}\quad \Psi_\alpha:=\frac{\alpha}{1+\alpha\kappa}{\hat{\sigma}}^2|\n \hat{u}|^2-{\hat{u}}^2. 
\end{equation}
Since the flux $\hat{\sigma} \frac{\partial \hat{u}}{\partial \nu}$ is continuous across $\partial \hat{E}$, one has necessarily $\Psi_\alpha=\mu$ on $\partial \hat{E}\backslash \partial \Omega$.

Now, let us follow the approach used in \cite{MuratTartar} and \cite{CasadoDiaz3} to simplify the writing of the optimality conditions. Notice first that $\hat{u}$ and $\hat{\sigma}^2 \left|\frac{\partial \hat{u}}{\partial \nu_{\hat{E}}}\right|^2$ are continuous across $\partial \hat{E}$. Let  $\n_\tau \hat{u}$ denote the tangential gradient of $\hat{u}$ on $\partial \hat{E}$.
 For the sake of clarity, the quantities computed on $\partial \hat{E}$ seen as the boundary of $\hat{E}$ will be denoted with the subscript $int$, whereas the ones computed on $\partial \hat{E}$ seen as part of the boundary of $\hat{E}^c$ will be denoted with the subscript $ext$. According to the optimality conditions \eqref{Eq:OptimalityH},  one has 
 \begin{align*}
\begin{split}
\left.\frac\alpha{1+\alpha\kappa} \hat{\sigma}^2 \left|\n_\tau \hat{u}\right|^2+\frac\alpha{1+\alpha\kappa} \hat{\sigma}^2\left(\frac{\partial \hat{u}}{\partial \nu}\right)^2-\hat{u}^2\right|_{int}
\leq \left.\frac\alpha{1+\alpha\kappa} \hat{\sigma}^2 \left|\n_\tau \hat{u}\right|^2+\frac\alpha{1+\alpha\kappa} \hat{\sigma}^2\left(\frac{\partial \hat{u}}{\partial \nu}\right)^2-\hat{u}^2\right|_{ext}
\end{split}\end{align*}
 on $\partial\hat{E}\backslash \partial \Omega$. By continuity of the flux $\hat{\sigma} \frac{\partial \hat{u}}{\partial \nu}$, we infer that 
$\alpha \hat{\sigma}^2 \left.\left|\n_\tau \hat{u}\right|^2\right|_{int}\leq \alpha \hat{\sigma}^2 \left.\left|\n_\tau \hat{u}\right|^2\right|_{ext}$ which comes to $(1+\alpha \kappa) \left.\left|\n_\tau \hat{u}\right|^2\right|_{int}\leq \left.\left|\n_\tau \hat{u}\right|^2\right|_{ext}$. Since $ \left.\left|\n_\tau \hat{u}\right|^2\right|_{int}= \left.\left|\n_\tau \hat{u}\right|^2\right|_{ext}$,  we have $\left.\n_\tau \hat{u}\right|_{\partial \hat{E}}=0$. Therefore, $\hat{u}$ is constant on $\partial \hat{E}\backslash \partial \Omega$ and since $\Psi_\alpha$ is constant on $\partial \hat{E}\backslash \partial \Omega$, it follows that $\left|\frac{\partial \hat{u}}{\partial \nu}\right|^2_{int}$ is constant as well on $\partial \hat{E}\backslash \partial \Omega$.

To {sum up}, the first order necessary conditions drive to the following condition:
\begin{equation}\label{Eq:Opt}
\text{The functions }\hat{u}\text{ and } |\n \hat{u}|\text{ are constant on $\partial \hat{E}\backslash \partial \Omega$.}
\end{equation}

\paragraph{Step 2: proof that $\O$ is necessarily a ball.} 
To prove that $\O$ is a ball, we will use Serrin's Theorem, that we recall hereafter.
\begin{thmnonumbering}\cite[Theorem 2]{Serrin1971}
Let $\mathscr E$ be a connected domain with a $\mathscr C^2$ boundary, $h$ a $\mathscr C^1(\R;\R)$ function and let $f\in \mathscr C^2\left(\overline{\mathscr E}\right)$ be a function satisfying
$$
-\Delta f= h(f)\,, \quad f>0 \text{ in }\mathscr E,\quad  f=0\text{ on }\partial \mathscr E,\quad \frac{\partial f}{\partial \nu}\text{ is constant on }\partial \mathscr E.$$
Then $\mathscr E$ is a ball and $f$ is radially symmetric.
\end{thmnonumbering}
According to \eqref{Eq:Opt}, let us introduce $\hat{\mu}=\left.\hat{u}\right|_{\partial \hat{E}}$.  One has $\hat{\mu}>0$ by using the maximum principle. Let us set $f=\hat{u}-\hat{\mu}$, $h(z)=\left(\lambda_\alpha(\hat{m})+\kappa\right)z$ and call $\mathscr E$ a given connected component of $\hat{E}$. By assumption, $\hat{E}$ is a $\mathscr C^2$ set, and, according to \eqref{Eq:Opt}, the function $\partial \left(\hat{u}-\hat{\mu}\right)/\partial \nu$ is constant on $\partial \hat{E}$.

The next result allows us to verify the last assumption of Serrin's theorem.

\begin{lemma}\label{Eq:Opt2}
There holds $\hat{u}>\hat{\mu}$  in $\hat{E}$.
\end{lemma}
 For the sake of clarity, the proof of this lemma is postponed to the end of this section.\\
 
Let us now come back to the proof that $\O$ is necessarily a ball. Take $x\in\partial \O$. Then $x$ belongs either to the closure of $\Omega \backslash \hat E$ or to the closure of $\hat{E}$. In the first case, there exists a connected component $\Gamma$ of $\partial \big(\O\backslash \hat{E}\big)$ which contains $x$. 
 
 Let us assume by contradiction that this connected component also intersects $\O$, then according to \eqref{Eq:Opt}, one has $\hat{u}$ constant on $\Gamma \cap \O$ and according to the Dirichlet boundary conditions, one has $\hat{u}=0$ on $\Gamma$, and hence $\hat{u}$ reaches its minimal value in the open set $\O$. According to the strong maximum principle, one gets that $\hat{u}(\cdot)=0$, and we have reached a contradiction. 
 
 Hence, $\Gamma \subset \partial\O$ and $\Gamma$ is connected. Similarly, if $x$ belongs to the closure of $\hat{E}$, then there exists a connected component of $\partial\hat{E}$ containing $x$, which is included in $\partial\O$. Hence, $\partial \O$ is the union of closed connected components of the boundaries of $\hat{E}$ and $\O\backslash \hat{E}$. As $\partial\O$ is connected by hypothesis, there only exists one such connected component, that we denote $\Gamma$. This implies in particular that if $\partial \hat E\cap \partial \O\neq \emptyset$, then $\partial (\O\backslash \hat E)\cap\partial \O=\emptyset$, and conversely, if $\partial\hat{E}^c \cap\partial \O\neq \emptyset$, then $\partial \hat E\cap \partial \O=\emptyset$.
  
{Assume first that the closure of $\O\backslash \hat{E}$ meets $\partial \O$. As $\hat{E}$ does not intersect $\partial \O$ in this case, one has $\hat{u}$ and $|\nabla \hat{u}|$ constant over the whole boundary of $\hat{E}$ by (\ref{Eq:Opt}) and thus Serrin's theorem applies: any connected component of $\hat{E}$ is a ball and $\hat{u}$ is radially symmetric over it.  We can hence fix a ball $\mathbb B(x_0;a_1)\subset \hat E$ such that $\hat u$ is radially symmetric in it.}

 Let us now consider the largest ball $\mathbb{B}(x_{0};a)\subset \O$ such that $\hat u$ is radially symmetric in $\mathbb B(x_0;a)$. If $\partial \mathbb B(x_0;a)\cap \partial \O\neq \emptyset$, then $\hat u=0$ on $\partial \mathbb B(x_0;a)$ which, by the maximum principle, implies that $\O=\mathbb B(x_0;a)$. We can hence assume that $\hat u>0$ on $\partial \mathbb B(x_0;a)$. Let us define 
 $$\mu_a:=\left.\frac{\alpha}{1+\alpha\kappa}{\hat{\sigma}}^2|\n \hat{u}|^2-{\hat{u}}^2\right|_{\partial \mathbb B(x_0;a)}.$$ If $\mu_a<\mu$ then by continuity of $\frac{\alpha}{1+\alpha\kappa}{\hat{\sigma}}^2|\n \hat{u}|^2-{\hat{u}}^2$, it follows that $m=0$ on $\mathbb B(x_0;a+\delta)\backslash \mathbb B(x_0;a-\delta)$. Thus, applying the Cauchy-Kovalevskaya theorem to all the tangential derivatives of $\hat{u}$ yields that $\hat{u}$ is radially symmetric in $\mathbb B(x_0;a+\delta)$ which is a contradiction with the definition of $a$. Indeed, the same arguments as \cite[Proof of Theorem~1, Part~2]{LamboleyLaurainNadinPrivat} would yield that if one writes $\hat{u}(x)=U(r)$, with $r:=|x-x_{0}|$, then there exists $\delta>0$ such that $U'(r)=0$ for all $r\in [a,a+\delta)$. Hence $u$ would remain constant on the annulus $\{a<|x-x_{0}<a+\delta\}$, yielding a contradiction.

The same reasoning yields the same contradiction if $\mu_a>\mu$, in which case $m=\kappa$ on $\mathbb B(x_0;a+\delta)\backslash \mathbb B(x_0;a-\delta)$. Finally, if $\mu_a=\mu$ then if follows from Lemma \ref{Eq:Opt2} that either $m=0$ in $\mathbb B(x_0;a)\backslash \mathbb B(x_0;a-\delta)$ and $m=\kappa$ in $\mathbb B(x_0;a+\delta)\backslash \mathbb B(x_0;a)$ or $m=\kappa$ in $\mathbb B(x_0;a)\backslash \mathbb B(x_0;a-\delta)$ and $m=0$ in $\mathbb B(x_0;a+\delta)\backslash \mathbb B(x_0;a)$. In both case, one can apply Carleman's unique continuation Theorem as done in \cite[Proof of Theorem 2.1, Step 3]{CasadoDiaz3} to conclude that $\hat u$ is radially symmetryc in $\mathbb B(x_0;\delta)$.
 
{In the case where $\partial\hat{E}$ meets $\partial \O$, then the boundary of $\O\backslash \hat{E}$ does not, and we conclude by using the same reasoning on $\O\backslash \hat{E}$ instead of $\hat{E}$, showing that $\hat{u}<\hat{\mu}$ on $\O\backslash \hat{E}$ and applying Serrin's theorem to $\hat{\mu}-\hat{u}$. }

%
%

\begin{proof}[Proof of Lemma \ref{Eq:Opt2}]
Let us set $v=\hat{u}-\hat{\mu}$, hence $v$ solves
\begin{equation}
\left\{\begin{array}{ll}
-\Delta v=\left(\lambda_\alpha(\hat{m})+\kappa\right)v+\left(\lambda_\alpha(\hat{m})+\kappa\right)\hat{\mu} & \text{ in }\hat{E},\\
v=0 \text{ on }\partial\hat{E}.
\end{array}\right.
\end{equation}
and we are led to show that $v>0$ in $\hat{E}$. Let $\lambda^D(\O)$ be the first Dirichlet eigenvalue\footnote{In other words 
\begin{equation}\label{Eq:LambdaD}\lambda^D(\O)=\inf_{u\in W^{1,2}_0( \O)\,,\int_{\O}u^2=1}\int_\O|\n u|^2>0.\end{equation}} of the Laplace operator in $E$.
By using the Rayleigh quotient \eqref{Eq:Rayleigh} we have
\begin{align*}
\lambda_\alpha(\hat{m})&= \min_{u\in W^{1,2}_0(\O)\,,u\neq 0}R_{\alpha,m}(u)
>  \min_{u\in W^{1,2}_0(\O)\,,u\neq 0}\frac{\frac12\int_\O|\n u|^2-\kappa\int_\O u^2}{\int_\O u^2}
=\lambda^D(\O)-\kappa,
\end{align*}
so that 
$\lambda_\alpha(\hat{m})+\kappa> \lambda^D(\O)>0$.
Now, since $v=0$ on $\partial \hat{E}$ and that $\hat{E}$ is a $\mathscr C^2$ open subset of $\O$, the extension  $\tilde v$ of $v$ by zero outside $\hat{E}$ belongs to $W^{1,2}_0(\O)$. Since $(\lambda_\alpha(\hat{m})+\kappa)$ and $\hat{\mu}$ are non-negative, we get
$$
-\Delta \tilde v\geq (\lambda_\alpha(\hat{m})+\kappa)\tilde v\text{ in }\hat{E}
$$
Splitting $\tilde v$ into its positive and negative parts as $\tilde v=\tilde v_+-\tilde v_-$ and multiplying the equation by $\tilde v_-$ we get after an integration by parts
$$-\int_{\O} |\n \tilde v_-|^2=-\int_{\hat{E}} |\n \tilde v_-|^2\geq -(\lambda_\alpha(\hat{m})+\kappa)\int_{\hat{E}}  \tilde v_-^2=-(\lambda_\alpha(\hat{m})+\kappa)\int_{\O}  \tilde v_-^2.$$
Using that $\lambda_\alpha(\hat{m})+\kappa> \lambda^D(\O)>0$, we get
$$\int_\O |\n \tilde v_-|^2\leq (\lambda_\alpha(\hat{m})+\kappa)\int_\O \tilde v_-^2<\lambda^D(\O)\int_\O \tilde v_-^2,
$$
which, combined with the Rayleigh quotient formulation of $\lambda^D(\O)$ yields $\tilde v_-=0$. Hence $v$ is nonnegative  in $\hat{E}$.
Using moreover that $(\lambda_\alpha(\hat{m})+\kappa)\geq 0$ and $\hat{\mu}\geq 0$ yields  that $-\Delta v\geq 0$  in $\hat{E}$
Notice that $v$ does not vanish identically in $\hat{E}$. Indeed, $u$ would otherwise be constant in $\hat{E}$ which cannot arise because of \eqref{Eq:EigenFunction}. According to the strong maximum principle, we infer that $v>0$ in $\hat{E}$.
\end{proof}

\begin{remark}
Following the arguments by Casado-Diaz in \cite{CasadoDiaz3}, it would be possible to weaken the regularity assumption on $E$ provided that we assume the stronger hypothesis that $\partial\O$ is simply connected. Indeed, in that case, assuming that $E$ is only of class $\mathscr C^1$ leads to the same conclusion. 
\end{remark}

\section{Proof of Theorem \ref{Th:RadialStability}}
Throughout this section, $\O$ will denote the ball $\mathbb B(0,R)$, which will also be denoted $\B$ for the sake of simplicity. Let $r^*_0\in (0,R)$ be chosen in such a way that 
$m^*_0=\kappa \mathds 1_{\mathbb B(0,r^*_0)}$ belongs to $\mathcal M_{m_0,\kappa}(\O).$ Let us introduce the notation $E_0^*=\B(0,r_0^*)$.

The existence part of { Theorem \ref{Th:RadialStability}} follows from a straightforward adaptation of \cite{ConcaMahadevanSanz}. In what follows, we focus on the second part of this theorem, that is, the stationarity of minimizers provided  $\alpha$ is small enough.
\subsection{Steps of  the proof for the stationarity } \label{sec:stepproofstationarity}
We argue by contradiction, assuming that, for any $\alpha>0$, there exists a radially symmetric distribution $\tilde m_{\alpha}$ such that $\lambda_\alpha(\tilde m_{\alpha})<\lambda_\alpha(m^*_0)$. Consider the resulting sequence $\{\tilde m_\alpha\}_{\alpha > 0}$.
\begin{itemize}
\item \textbf{Step 1:} we prove that $\{\tilde m_\alpha\}_{\alpha \to 0}$ converges strongly to $m^*_0$ in $L^1$, as $\alpha \to 0$. Regarding the associated eigenfunction, we prove that $\{u_{\alpha,m_\alpha}\}_{\alpha> 0}$ converges strongly to $u_{0,m^*_0}$ in $\mathscr C^0$ and that $\alpha \n u_{\alpha, m_\alpha}$ converges to 0 in $L^\infty(\B)$, as $\alpha\to 0$.
\item \textbf{Step 2:} by adapting \cite[Theorem 3.7]{Laurain}, we  prove that we can {restrict} ourselves to considering {\it bang-bang}  radially symmetric distributions of resources $\tilde m_{\alpha}=\kappa \mathds 1_{\tilde E}$ such that the Hausdorff distance $d_H(\tilde E,E_0^*)$ is arbitrarily small.
\item \textbf{Step 3:} this is the main innovation of the proof.
Introduce $h_\alpha=\tilde m_\alpha-m^*_0$, and consider the path $\{m_t\}_{t\in [0,1]}$ from $m_\alpha$ to $m^*_0$ defined by $m_t=m_0^*+th_\alpha$. We then consider the mapping 
$$
f_\alpha:t\mapsto \zeta_\alpha(m_t,\Lambda_-(m_t))
$$ 
where $\zeta_\alpha$ and $\Lambda_-(m_t)$ are respectively given by { equation~\eqref{Eq:Zeta} and definition~\ref{De:Homo}}.  Notice that, since $m^*_0$ and $\tilde m_\alpha$ are {\it bang-bang}, $f_\alpha(0)=\lambda_\alpha(m^*_0)$ and $f_\alpha(1)=\lambda_\alpha(m_\alpha)$ according to Def. \ref{De:Homo}. Let $u_t$ be a $L^2$ normalized eigenfunction associated with $(m_t,\Lambda_-(m_t))$, in other words a solution to the equation 
\begin{equation}\label{Eq:Ut}
\left\{\begin{array}{ll}
-\n\cdot\left( \Lambda_-(m_t)\n u_t\right)=\zeta_\alpha(m_t,\Lambda_-(m_t)) u_t+m_tu_t & \text{ in }\B\\
u_t=0 & \text{ on }\partial \B\\
\int_\B u_t^2=1. & 
\end{array}\right.
\end{equation}
According to the proof of the optimality conditions \eqref{Eq:OptimalityH}, one has 
$$
f_\alpha'(t)=\int_\B h_\alpha\left(\frac\alpha{1+\alpha\kappa} \Lambda_-(m_t)^2|\n u_t|^2-u_t^2   \right).
$$
Applying the mean value theorem yields the existence of $t_1\in (0,1)$ such that $\lambda_\alpha(\tilde m_\alpha)-\lambda_\alpha(m^*_0)={f_{\alpha}'(t_1)}$. This enables us to show that, for $t\in [0,1]$ and $\alpha$ small enough, one has
$$
f_\alpha'(t)\geq C \int_\B |h_\alpha| \operatorname{dist}(\cdot, \S)
$$ for some $C>0$, giving in turn $\lambda_\alpha(m_\alpha)-\lambda_\alpha(m^*_0)\geq C \int_\B |h_\alpha| \operatorname{dist}(\cdot, \S).$ (we note that the same quantity is obtained in \cite{Laurain}. Nevertheless, we obtain it in a more straightforward manner which bypasses the exact decomposition of eigenfunctions and eigenvalues used there.).
\end{itemize}

Let us now provide the details of each step.

\subsection{Step 1: convergence of quasi-minimizers and of sequences of eigenfunctions}
We first investigate the convergence of quasi-minimizers.
\begin{lemma}\label{Le:Convergence}
Let $\{m_\alpha\}_{\alpha>0}$ be a sequence in $\mathcal M_{m_0,\kappa}(\O)$ such that,  
\begin{equation}\label{eq:lem:lambalphaineg}
\forall \alpha>0, \quad \lambda_\alpha(m_\alpha)\leq \lambda_\alpha(m^*_0).
\end{equation}
Then, $\{m_\alpha\}_{\alpha>0}$ converges strongly to $m^*_0$ in $L^1(\O)$.
\end{lemma}
\begin{proof}[Proof of Lemma \ref{Le:Convergence}]
The sequence $(\lambda_\alpha(m_\alpha))_{\alpha>0}$ is bounded from above. Indeed, choosing any test function $\psi\in W^{1,2}_0(\O)$ such that $\int_\O \psi^2=1$, it follows from \eqref{Eq:Rayleigh} that $\lambda_\alpha(m_\alpha)\leq (1+\alpha \kappa)\Vert \n \psi\Vert_2^2+\kappa \Vert \psi\Vert_2^2.$
Similarly, using once again \eqref{Eq:Rayleigh}, we get that if $\xi_\alpha$ is the first eigenvalue associated to the operator $-(1+\alpha \kappa)\Delta-\kappa $, then $\lambda_\alpha(m_\alpha)\geq \xi_\alpha$. 
Since $(\xi_\alpha)_{\alpha>0}$ converges to the first eigenvalue of $-\Delta -\kappa$ as $\alpha \to0$, $(\xi_\alpha)_{\alpha>0}$ is bounded from below whenever $\alpha$ is small enough. Combining these facts yields that the sequence $(\lambda_\alpha(m_\alpha))_{\alpha> 0}$ is bounded by some positive constant $M$ and converges, up to a subfamily, to $\tilde \lambda$.
For any $\alpha>0$, let us denote by $u_\alpha$ the associated $L^2$-normalized eigenfunction associated to  $\lambda_\alpha(m_\alpha)$. From the weak formulation of  equation \eqref{Eq:EigenFunction} and the normalization condition $\int_\O u_\alpha^2=1$, we infer that
\begin{align*}
\Vert \n u_\alpha\Vert _2^2&=\int_\O |\n u_\alpha|^2\leq \int_\O (1+\alpha m)|\n u_\alpha|^2=\int_\O m_\alpha u_\alpha^2+\lambda_\alpha(m_\alpha)\int_\O u_\alpha^2 \leq (M+\kappa).
\end{align*}
According to the Poincar\'e inequality and the Rellich-Kondrachov Theorem, the sequence $(u_\alpha)_{\alpha >0}$ is uniformly bounded in $W^{1,2}_0(\O)$ and converges, up to subfamily, to $\tilde u\in W^{1,2}_0(\O)$ weakly in $W^{1,2}_0(\O)$ and strongly in $L^2(\O)$, and moreover $\tilde u$ is also normalized in $L^2(\O)$.

 Furthermore, since $L^2$ convergence implies pointwise convergence (up to a subfamily), $\tilde u$ is necessarily nonnegative in $\O$. Let $\tilde m$ be a closure point of $(m_\alpha)_{\alpha >0}$ for the weak-star topology of $L^\infty$. Passing to the weak limit in the weak formulation of the equation solved by $u_\alpha$, namely Eq.~\eqref{Eq:EigenFunction}, one gets
$$
-\Delta \tilde u-\tilde m\tilde u=\tilde \lambda \tilde u\quad \text{in }\O.
$$
Since $\tilde u\geq 0$ and $\int_{\B(0,R)}\tilde u^2=1$, it follows that $\tilde u$ is the principal eigenfunction of $-\Delta -\tilde m$, so that $\tilde \lambda=\lambda_0(m^*_0)$.

Mimicking this reasoning enables us to show in a similar way that, up to a subfamily, $(\lambda_\alpha(m^*_0))_{\alpha > 0}$ converges to $\lambda_0(m^*_0)$ and $(u_{\alpha,m^*_0})_{\alpha > 0}$ converges to $u_{0,m^*_0}$ as $\alpha \to 0$.  Passing to the limit in the inequality \eqref{eq:lem:lambalphaineg} and since $m^*_0$ is the only minimizer of $\lambda_0$ in $\mathcal M_{m_0,\kappa}(\O)$ according to the Faber-Krahn inequality, we infer that necessarily, $\tilde m=m^*_0$.  Moreover, $m^*_0$ being an extreme point of $\mathcal M_{m_0,\kappa}(\O)$, the subfamily $(m_\alpha)_{\alpha > 0}$ converges to $\tilde m=m^*_0$ (see \cite[Proposition 2.2.1]{HenrotPierre}), strongly in $L^1(\O)$.
\end{proof}

A straightforward adaptation of the proof of Lemma \ref{Le:Convergence} yields that both sets $\{\lambda_\alpha(m)\}_{m\in \mathcal M_{m_0,\kappa}(\O)}$ and $\{\Vert u_{\alpha,m}\Vert_{W^{1,2}(\O)}\}_{m\in \mathcal M_{m_0,\kappa}(\O)}$ are uniformly bounded whenever $\alpha \leq 1$.
Let us hence introduce $M_0>0$ such that 
\begin{equation}\label{metz1946}
\forall \alpha\in [0,1], \quad \max \{|\lambda_\alpha(m)|, \Vert \uam\Vert _{W^{1,2}_0(\O)}\}\leq M_0.
\end{equation}

The next result is the only ingredient of the proof of Theorem~\ref{Th:RadialStability} where the low dimension assumption on $n$ is needed.

\begin{lemma}\label{Le:Bounds}
Let us assume that $n=1,2,3$. There exists $M_1>0$ such that, for every radially symmetric distribution $m\in \mathcal M_{m_0,\kappa}(\O)$ and every $\alpha\in  [0,1]$, there holds
$$
\Vert u_{\alpha,m}\Vert _{W^{1,\infty}(\B)}\leq M_1.
$$
Furthermore,  define $\tsm$, $\tm$ and  $\pam:(0,R)\to \R$ by 
$$\forall x \in \B,\quad  \uam(x)=\pam\left(|x|\right),\quad \sigma_{\alpha,m}(x)=\tsm(|x|),\quad  m(x)=\tm(|x|),$$
then $\tsm(\pam)'$ belongs to $W^{1,\infty}(0,R)$. 
\end{lemma}
\begin{proof}[Proof of Lemma \ref{Le:Bounds}]
This proof is inspired by \cite[Proof of Theorem 3.3]{Laurain}. It is standard that for every $\alpha\in [0,1]$ and every radially symmetric distribution $m\in \mathcal M_{m_0,\kappa}(\O)$, the eigenfunction $u_{\alpha,m}$ is itself radially symmetric. 
By rewriting the equation \eqref{Eq:EigenFunction} on $ \uam$ in {spherical} coordinates, on sees that $\pam$ solves
\begin{equation}\label{Eq:EFRadial}
\left\{
\begin{array}{ll}
-\frac{d}{dr}\left(r^{n-1}\tsm \frac{d}{dr}\pam\right)=\left(\lambda_\alpha(m)\pam+\tm \pam\right)r^{n-1} & \text{in }(0,R)\\
\pam(R)=0. & 
\end{array}
\right.
\end{equation}
By applying the Hardy Inequality\footnote{This inequality reads (see e.g.  \cite[Lemma 1.3]{AloisKufner} or \cite{Hardy}): for any non-negative $f$,
$$
\int_0^\infty f(x)^2dx\leq 4\int_0^\infty x^2f'(x)^2dx.
$$}
to $f=\pam$, we get 
{\begin{eqnarray*}
\int_0^R \p^{2}_{\alpha,m}(r)\, dr & \leq & 4\int_0^R r^2\p'_{\alpha,m}(r)^2\, dr \\
& \leq & 4R^2\int_0^R \left(\frac{r}{R^{2}}\right)^{n-1}\p'_{\alpha,m}(r)^2dr=4R^{4-2n}\Vert\n \uam\Vert_{L^2(\B)}^2\leq M,
\end{eqnarray*}}
since $n\in \{1,2,3\}$. Hence, there exists $C>0$ such that 
\begin{equation}\label{Borne1D}
\Vert \pam\Vert _{L^2(0,R)}^2\leq C.
\end{equation}

We will successively prove that $\pam$ is uniformly bounded in $W^{1,2}_0(0,R)$, then in $L^\infty(0,R)$ to infer that  $\p'_{\alpha,m}$ is bounded in $L^\infty(0,R)$. This proves in particular that $\sigma_{\alpha,m}\p_{\alpha,m}'\in L^{\infty}(0,R)$.
We will then conclude that $\sigma_{\alpha,m}\p_{\alpha,m}'\in W^{1,\infty}(0,R)$ by using  that it is a continuous function whose derivative is uniformly bounded in $L^\infty$ by the equation on $\pam$.\

According to \eqref{metz1946}, {one sees that} $\Vert r^{\frac{n-1}2}\p'_{\alpha,m}\Vert _{L^2(0,R)}=\Vert\n \uam\Vert_{L^2(\B)}$ is bounded and therefore, $r^{n-1}\p'_{\alpha,m}(r)$ converges  to 0 as $r\to 0$. {If $n=1$, such a convergence holds since $\p'_{\alpha,m}(0)=0$ by radial symmetry}. Hence, integrating Eq. \eqref{Eq:EFRadial} between $0$ and $r>0$ yields
$$
\tsm(r)\p'_{\alpha,m}(r)=-\frac1{r^{n-1}}\int_0^r t^{n-1}\left(\lambda_\alpha(m)\pam(t)+\tm(t) \pam(t)\right) dt.
$$
By using the Cauchy-Schwarz inequality and  \eqref{Borne1D}, we get the existence of $\tilde M>0$ such that
\begin{eqnarray*}
\Vert \p'_{\alpha,m}\Vert _{L^2(0,R)}^2 &\leq &\int_0^1 \left(\tsm \p'_{\alpha,m}\right)^2(r)\, dr\\
&=&\int_0^1\frac1{r^{2(n-1)}}\left(\int_0^r t^{n-1}\left(\lambda_\alpha(m)\pam(t)+\tm(t) \pam(t)\right) dt\right)^2\, dr\\
&\leq &\int_0^1\frac1{r^{2(n-1)}}(\lambda_\alpha(m)+\kappa)^2\Vert \pam\Vert _{L^2(0,R)}^2\Vert {t\mapsto  t^{n-1}}\Vert _{L^2(0,r)}^2\, dr\\
&\leq &\frac{\tilde M}{4n-2}\Vert \pam\Vert _{L^2(0,R)}^2\leq \tilde M\Vert \pam\Vert _{L^2(0,R)}^2\leq \tilde M C,
\end{eqnarray*}
Hence, $\pam$ is uniformly bounded in $W^{1,2}_0(0,R)$.

It follows from standard Sobolev embedding's theorems that there exists a constant $M_2>0$, such that 
$\Vert \pam\Vert _{L^\infty(0,R)}\leq M_2$. 

Finally, plugging this estimate in the equality 
$$
\tsm(r)\p'_{\alpha,m}(r)=-\frac1{r^{n-1}}\int_0^r t^{n-1}\left(\lambda_\alpha(m)\pam(t)+\tm(t) \pam(t)\right) dt
$$
and since $t^{n-1}\leq r^{n-1}$ on $(0,r)$, we get that $\p'_{\alpha,m}$ is uniformly bounded in $L^\infty(0,R)$.

\end{proof}

The next lemma is a direct corollary of Lemma \ref{Le:Convergence}, Lemma \ref{Le:Bounds} and the Arzela-Ascoli Theorem.
\begin{lemma}\label{Le:ConvergenceFP}
Let $(m_\alpha)_{\alpha>0}$ be a sequence of radially symmetric functions of $\mathcal M_{m_0,\kappa}(\O)$ such that, for every $\alpha\in [0,1]$, $\lambda_{\alpha}(m_\alpha)\leq \lambda_{\alpha}(m^*_0)$. Then, up to a subfamily,
$u_{\alpha,m_\alpha}$ converges to $u_{0,m^*_0}$ for the strong topology of $\mathscr C^0(\overline\O)$ as $\alpha\to 0$.
\end{lemma}

\subsection{Step 2: reduction to particular resource distributions close to $m_0^*$}
Let us consider a sequence of radially symmetric distributions $(m_\alpha)_{\alpha > 0}$ such that, for every $\alpha\in [0,1]$, $\lambda_\alpha(m_\alpha)\leq \lambda_\alpha(m^*_0)$.
According to Proposition~\ref{Th:BangBang}, we can assume that each $m_\alpha$ is a {\it bang-bang}, in other words that $m_\alpha =\kappa \mathds{1}_{E_\alpha}$ where $E_\alpha$ is a measurable subset of $\B(0,R)$. For every $\alpha \in [0,1]$, one introduces $d_\alpha=d_H(E_\alpha,E^*_0)$, the Hausdorff distance of $E_\alpha$ to $E^*_0$.
\begin{lemma}\label{Le:Hausdorff}
For every $\e>0$ small enough, there exists $\overline \alpha>0$ such that, for every $\alpha\in [0, \overline \alpha]$, there exists a {radially symmetric } measurable subset $\tilde{E}_\alpha$ of $\O$ such that
$$
\lambda_\alpha(\kappa\mathds{1}_{E_\alpha})\geq \lambda_\alpha(\kappa\mathds{1}_{\tilde{E}_\alpha}), \quad |E_\alpha|=|\tilde{E}_\alpha|\quad \text{and}\quad d_H(\tilde{E}_\alpha,E^*_0) \leq \e.
$$
\end{lemma}
\begin{proof}[Proof of Lemma \ref{Le:Hausdorff}]
Let $\alpha\in [0,1]$. Observe first that
$\lambda_\alpha(m)=\int_\B|\n \uam|^2+\alpha \int_\B m|\n \uam|^2-\int_\B mu_{\alpha,m}^2=\int_\B |\n \uam|^2+\int_\B m\psi_{\alpha,m}$, where $\psi_{\alpha,m}$ has been introduced in Lemma~\ref{Le:DeriveeL21}.
We will first  construct $ \tilde m_\alpha$ in such a way that 
$$
\lambda_\alpha(m_\alpha)\geq \int_\B |\n u_{\alpha,m_\alpha}|^2+\int_\O \psi_{\alpha,m_\alpha} \tilde m_\alpha\geq \lambda_\alpha(\tilde m_\alpha),
$$ 
and, to this aim, we will define $\tilde m_\alpha$ as a suitable level set of $\psi_{\alpha,m_\alpha}$. Thus, we will evaluate the Hausdorff distance of these level sets to $E^*_0$. The main difficulty here rests upon the lack of regularity of the switching function $\psi_{\alpha,m_\alpha}$, which is { not even continuous (see Figure (\ref{fig:phi})).}

According to Lemmas~\ref{Le:Bounds} and \ref{Le:ConvergenceFP}, $\psi_{\alpha,m_\alpha}$ converges to $-u_{0,m^*_0}^2$ for the strong topology of $L^\infty(\B)$. 
Recall that $m^*_0=\kappa \mathds{1}_{\mathbb B(0,r^*_0)}$ and let $V_0$ be defined by $V_0=|\mathbb B(0,r^*_0)|$.
Let us define $\mu_\alpha^*$ by dichotomy, as the only real number such that
$$
|\underline \omega_\alpha|\leq V_0\leq |\overline \omega_\alpha|,
$$
 where $\underline \omega_\alpha=\{\psi_{\alpha,m_\alpha}<\mu_\alpha^*\}$ and $\overline \omega_\alpha=\{\psi_{\alpha,m_\alpha}\leq \mu_\alpha^*\}$.

Since $\left|\{\psi_{0,m^*_0}<-\varphi_{0,m_0^*}^2(r_0^*)\}\right|=V_0$, we deduce that $(\mu_\alpha^*)$ converges to $-\varphi_{0,m^*_0}^2(r^*_0)$ as $\alpha\to 0$. 
Since $\varphi_{0,m^*_0}$ is decreasing, we infer that for any $\e>0$ small enough, there exists $\overline \alpha>0$ such that: for every $\alpha\in [0, \overline \alpha]$, $\mathbb B(0,r^*_0-\e)\subset \underline \omega_\alpha\subset \overline \omega_\alpha\subset \mathbb B(0,r^*_0+\e)$.
Therefore, there exists a radially symmetric set $B_\e^\alpha$ such that
$$
\underline \omega_\alpha \subset B_\e^\alpha\subset \overline \omega_\alpha,\quad  |B_\e^\alpha|=V_0 ,\quad  d_H(B_\e^\alpha,E^*_0)\leq \e.
$$ 
Since $E_\alpha$ and $B_\e^\alpha$ have the same measure, one has $|(E_\alpha)^c\cap B_\e^\alpha|=|E_\alpha\cap (B_\e^\alpha)^c|$, we introduce $\tilde m_\alpha=\kappa \mathds{1}_{B_\e^\alpha}$ so that $\tilde m_\alpha$ belongs to $\mathcal M_{m_0,\kappa}(\O)$.

\begin{figure}[H]\label{fig:phi}
\begin{center}
\includegraphics[width=5.2cm]{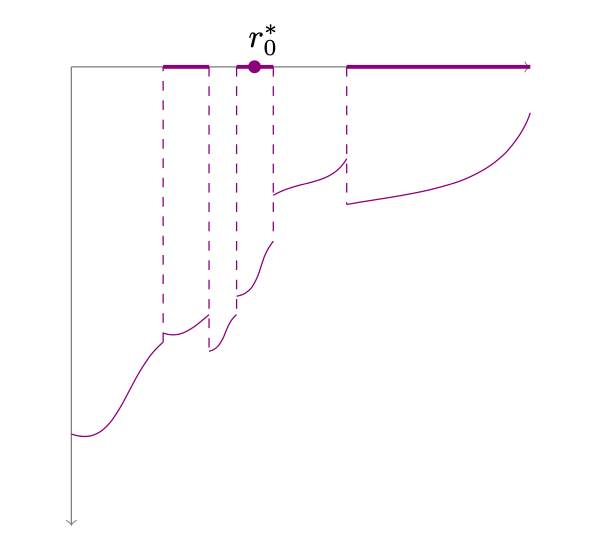}
\caption{Possible graph of the discontinuous function $\psi_{\alpha,m_\alpha}$. The bold intervals on the $x$ axis correspond to $\{m_\alpha=0\}$. 
}
\end{center}
\end{figure}

%
%
%
%
%

By construction, one has 
\begin{align*}
\lambda_\alpha(m_\alpha)&=\int_\B (1+\alpha m_\alpha)|\n u_{\alpha,m_\alpha}|^2-\int_\B m_\alpha u_{\alpha,m}^2
=\int_\B |\n u_{\alpha,m_\alpha}|^2+\int_\B \psi_{\alpha,m_\alpha}m_\alpha
\\&=\int_\B |\n u_{\alpha,m_\alpha}|^2+\kappa \int_{E_\alpha}\psi_{\alpha,m_\alpha}
=\int_\B |\n u_{\alpha,m_\alpha}|^2+\kappa \int_{E_\alpha\cap (B_\e^\alpha)^c}\psi_{\alpha,m_\alpha}+\kappa \int_{E_\alpha\cap B_\e^\alpha}\psi_{\alpha,m_\alpha}
\\&\geq \int_\B |\n u_{\alpha,m_\alpha}|^2+\kappa \mu_\alpha^* |E_\alpha \cap( B_\e^\alpha)^c|+\kappa \int_{E_\alpha\cap B_\e^\alpha}\psi_{\alpha,m_\alpha}
\\&=\int_\B |\n u_{\alpha,m_\alpha}|^2+\kappa \mu_\alpha^* |(E_\alpha)^c \cap B_\e^\alpha|+\kappa \int_{E_\alpha\cap B_\e^\alpha}\psi_{\alpha,m_\alpha}
\\&\geq \int_\B |\n u_{\alpha,m_\alpha}|^2+\kappa \int_{(E_\alpha)^c\cap B_\e^\alpha}\psi_{\alpha,m_\alpha}+\kappa \int_{E_\alpha\cap B_\e^\alpha}\psi_{\alpha,m_\alpha}
=\int_\B |\n u_{\alpha,m_\alpha}|^2+\int_\B \tilde m_\alpha \psi_{\alpha,m_\alpha}
\\&=\int_\B\sigma_{\alpha,\tilde m_\alpha} |\n u_{\alpha,m_\alpha}|^2-\int_\B \tilde m_\alpha u_{\alpha,m}^2
\geq \lambda_\alpha(\tilde m_\alpha),
\end{align*}
the last inequality coming from the variational formulation \eqref{Eq:Rayleigh}.
The expected conclusion thus follows  by taking $\tilde E_\alpha:=B_\e^\alpha$.
\end{proof}

From now on we will replace $m_\alpha$ by $\kappa\mathds{1}_{\tilde{E}_\alpha}$ and still denote this function by $m_\alpha$ with a slight abuse of notation.

\subsection{Step 3: conclusion, by the mean value theorem}
Recall that, according to Section~\ref{sec:stepproofstationarity}, for every $\alpha\in [0,1]$, the mapping $f_\alpha$ is defined by $f_\alpha(t):=\zeta_\alpha(m_t,\Lambda_-(m_t))$ for all $t\in [0,1]$
We claim that $f_\alpha$ belongs to $\mathscr C^1$. This follows from similar arguments to those of the $L^2$ differentiability of $m\mapsto \lambda_\alpha(m)$ in Appendix \ref{Ap:Differentiability}. Following the proof of \eqref{Eq:OptimalityH}, it is also straightforward that for every $t\in [0,1]$, one has
\begin{equation}\label{Eq:DeriveeFAlpha}
f_\alpha'(t)=\int_\B \left(\frac\alpha{1+\alpha\kappa} \Lambda_-(m_t)^2 |\n u_t|^2-u_t^2\right)h_\alpha.\end{equation} 
Finally, since $m^*_0$ and $m_\alpha$ are {\it bang-bang}, it follows from Definition~\ref{De:Homo} that 
$f_\alpha(0)=\lambda_\alpha(m^*_0)$ and $f_\alpha(1)=\lambda_\alpha(m_\alpha)$.

Since $m_\alpha$ is assumed to be radially symmetric, so is $m_t$ for every $t\in [0,1]$ thanks to a standard reasoning, and, therefore, so is  $u_t$. With a slight abuse of notation, we identify $m_t$, $u_t$ and $\Lambda_-(m_t)$ with their radially symmetric part $\tilde m_t$, $\tilde u_t$, $\tilde \Lambda_-(m_t)$ defined on $[0,R]$ by 
$$
u_t(x)=\tilde u_t(|x|), \quad m_t(x)=\tilde m_t(|x|),\quad  \Lambda_-(m_t)(x)=\tilde \Lambda_-(\tilde m_t)(|x|).
$$
Then the function $u_t$ (defined on $[0,R]$) solves the equation
\begin{equation}\label{Eq:UtRadial}
\left\{\begin{array}{ll}-\frac{d}{dr}\left(r^{n-1}\Lambda_-(m_t)\frac{du_t}{dr}\right)=\left(\zetat u_t+m_t u_t\right)r^{n-1}& r\in [0,R]\\ 
u_t(R)=0 &\\
\int_0^R r^{n-1}u_t(r)^2dr=\frac1{c_n},&
\end{array}\right.
\end{equation} 
where $c_n=|\mathbb S(0,1)|$. As a consequence, an immediate adaptation of the proof of Lemma~\ref{Le:Bounds} yields:
\begin{lemma}\label{Le:BoundsUt}
There exists $M>0$ such that
$$
\max \left\{\Vert u_t\Vert _{W^{1,\infty}},\Big\Vert \Lambda_-(m_t)u_t'\Big\Vert_{W^{1,\infty}}\right\}\leq M.
$$
Furthermore, $\Lambda_-(m_t)u_t'$ converges to $u_{0,m^*_0}'$ in $L^\infty(0,R)$ and uniformly with respect to $t\in [0,1]$, as $\alpha\to 0$.
\end{lemma}

According to the mean value theorem, there exists $t_1=t_1(\alpha)\in [0,1]$ such that 
$$
\lambda_\alpha(m_\alpha)-\lambda_\alpha(m^*_0)=f_\alpha(1)-f_\alpha(0)=f_\alpha'(t_1)
$$ 
and by using  Eq. \eqref{Eq:DeriveeFAlpha}, one has
$$
f_\alpha'(t_1)=\int_\B \left(\frac\alpha{1+\alpha\kappa} \Lambda_-(m_{t_1})^2 |\n u_{t_1}|^2-u_{t_1}^2\right)h_\alpha,
$$
where $h_\alpha=m_\alpha-m^*_0$.
Let us introduce $I_\alpha^\pm$ as the two subsets of $ [0,R]$ given by $I_\alpha^\pm=\{h_\alpha=\pm1\}$. Let $\e>0$. According to Lemma~\ref{Le:Hausdorff},  we have, for $\alpha$ small enough, 
$I_\alpha^+\subset [r^*_0,r^*_0+\e]$ and $ I_\alpha^-\subset [r^*_0-\e;r^*_0]$.
Finally, let us introduce
$$
\mathfrak F_1:=\frac\alpha{1+\alpha\kappa} \Lambda_-(m_{t_1})^2 |\n u_{t_1}|^2-u_{t_1}^2.
$$
According to Lemma~\ref{Le:BoundsUt}, $\mathfrak F_1$ belongs to $W^{1,\infty}$ and $ \mathfrak F_1+u_{\alpha,m^*_0}^2$ converges to $0$ as $\alpha\to 0$, for the strong topology of $W^{1,\infty}(0,R)$.
Moreover, there exists $M>0$ independent of $\alpha$ such that for $\e>0$ small enough,
$$
-M\leq 2 u_{\alpha,m^*_0}\frac{du_{\alpha,m^*_0}}{dr}\leq-M\quad  \text{ in } [r^*_0-\e;r^*_0+\e]
$$ 
and it follows that 
$$
\frac{M}2\leq \frac{d\mathfrak F_1}{dr}\leq 2M\quad  \text{ in }[r^*_0-\e;r^*_0+\e]$$ for $\alpha$ small enough.
Hence, since $\mathfrak F_1$ is Lipschitz continuous and thus absolutely continuous, one has for every $y\in [0,\e]$,
\begin{align*}
\mathfrak F_1(r^*_0+y)&=\mathfrak F_1(r^*_0)+\int_{r^*_0}^{r^*_0+y} \mathfrak F_1'(s)\, ds\geq \mathfrak F_1(r^*_0)+\frac{M}2 y \\ 
\text{and}\quad \mathfrak F_1(r^*_0-y)&=\mathfrak F_1(r^*_0)+\int_{r^*_0-y}^{r^*_0} (-\mathfrak F_1'(s))\, ds\leq \mathfrak F_1(r^*_0)-\frac{M}{2} y.
\end{align*}
Since $h_\alpha \leq 0 $ in $[r^*_0-\e;r^*_0]$ and $h_\alpha\geq 0$ in $ [r^*_0,r^*_0+\e]$, we have 
\begin{eqnarray*}
h_\alpha(r^*_0+y)\mathfrak F_1(r^*_0+y)&\geq &h_\alpha(r^*_0+y)\mathfrak F_1(r^*)+\frac{|h_\alpha| (r^*_0+y)M}2 y\\
\text{and}\quad h_\alpha(r^*_0-y)\mathfrak F_1(r^*_0-y)&\geq& h_\alpha(r^*_0-y)\mathfrak F_1(r^*)+\frac{|h_\alpha|(r^*_0-y) M}2 y.
\end{eqnarray*}
for every $y\in [0,\e]$. Hence, using that $\int_\B h_\alpha=0$, we infer that
\begin{align*}
f_\alpha'(t_1)&=\int_\B \left(\frac\alpha{1+\alpha\kappa}\Lambda_-(m_{t_1})^2 |\n u_{t_1}|^2-u_{t_1}^2\right)h_\alpha =c_{n}\int_0^R h_\alpha(s) \mathfrak F_1(s) s^{n-1}\, ds
\\&=c_n \left(\int_{r^*_0-\e}^{r^*_0}h_\alpha\mathfrak F_1(s)s^{n-1}ds+\int_{r^*_0}^{r^*_0+\e} h_\alpha\mathfrak F_1(s)s^{n-1}\, ds\right)
\\&\geq c_n\left(\int_{r^*_0-\e}^{r^*_0}h_\alpha(s)\mathfrak F_1(r^*)s^{n-1}ds+\int_{r^*_0}^{r^*_0+\e} h_\alpha(s)\mathfrak F_1(r^*)s^{n-1}\, ds\right)&
\\&+\frac{c_nM}2 \left(\int_{r^*_0-\e}^{r^*_0}|h_\alpha|(s) |r^*_0-s|s^{n-1}ds+\int_{r^*_0}^{r^*_0+\e} |h_\alpha|(s)|r^*_0-s|s^{n-1}\, ds\right)
\\&=\frac{c_nM}2\int_\B |h_\alpha|\operatorname{dist}(\cdot,\S),
\end{align*}
which concludes Step 3. Theorem \ref{Th:RadialStability} is thus proved.

\begin{remark}\label{Susu:Concluding}
Regarding the proof of Theorem \ref{Th:RadialStability}, it would have been more natural to consider the path $t\mapsto \left(\lambda_\alpha(m_t),m_t\right)$ rather than $t\mapsto \left( \zetat,m_t\right)$. However, we would have been led to consider {$\mathfrak G_1=\alpha \kappa |\n u_{\alpha,m_{t_{1}}}|^2-u_{\alpha,m_{t_{1}}}^2$ instead of $\mathfrak F_1$}. Unfortunately, this would have been more intricate because of the regularity of $\mathfrak G_1$, which is discontinuous and thus, no longer a $W^{1,\infty}$ function, so that a Lemma analogous to Lemma \ref{Le:BoundsUt} would not be true. Adapting step by step the arguments of \cite{Laurain} would nevertheless be possible although much more technical. 
\end{remark}


\section{Sketch of the proof of Corollary \ref{Th:Sketch}}\label{Se:Sketch}

We do not give all details since the proof is then very similar to the ones written previously. We only underline the slight differences in every step.

To prove this result, we consider the following relaxation of our problem, which is reminiscent of the problems considered in \cite{Hamel2011}. Let us consider, for any pair $(m_1,m_2)\in \mathcal M_{m_0,\kappa}(\O)^2$, the first eigenvalue of the operator $\mathscr N:u\mapsto -\n \cdot\left((1+\alpha m_1)\n u\right)-m_2 u$, and write it $\eta_\alpha(m_1,m_2)$. Let $m^*_0:=\kappa \mathds 1_{\mathbb B(0,R)}$.
By using the results of \cite{Hamel2011} or alternatively, applying the rearrangement of Alvino and Trombetti, \cite{AlvinoTrombetti} as it has been done in \cite{ConcaMahadevanSanz}, one proves the existence of a radially symmetric function $\tilde m_1$ such that 
$$\eta_\alpha(m_1,m_2)\geq \eta_\alpha(\tilde m_1,m^*_0),$$
so that we are done if we can prove that, for any $m\in \mathcal M(\O)$ there holds
\begin{equation}\label{Eq:Sk}
\eta_\alpha(m,m^*_0)\geq \eta_\alpha(m^*_0,m^*_0).
\end{equation}
We claim that \eqref{Eq:Sk} holds for any  $m\in \mathcal M_{m_0,\kappa}$, provided that $m_0$ and $\alpha$ be small enough. Let us describe the main steps of the proof:
\begin{itemize}
\item {\bf Step 1:} mimicking the compactness argument used in \cite{ConcaMahadevanSanz}, one shows that there exists a solution $m_\alpha$ to the problem
$$\inf_{m\in \mathcal M_{m_0,\kappa}(\O)}\eta_\alpha(m,m^*_0),$$
which is radially symmetric and {\it bang-bang}. We write it $m_\alpha=\kappa \mathds 1_{E_\alpha}$.
\item {\bf Step 2:}  let $\mu_0$ and $r_0^*$ be the unique real numbers such that 
$$
\left|\left\{|\n u_{0,m^*_0}|^2\leq \mu_0\right\}\right|=V_0=|\mathbb B(0,r^*_0)|.
$$ 
Introducing  $E_0=\left\{|\n u_{0,m^*_0}|^2\leq \mu_0\right\}$, we prove that
$m_\alpha$ converges in $L^1(\O)$ to $\kappa \mathds 1_{E_0}$ as $\alpha \to 0$.
\item {\bf Step 3:} we establish that if $m_0$ is small enough, then $E_0=\mathbb B(0,r^*_0)$. This is done by proving that $u_{0,m^*_0}$ converges in $\mathscr C^1$ to the first Dirichlet eigenfunction of the ball as $r^*_0\to 0$ and by determining the level-sets of this first eigenfunction, as done in \cite[Section 2.2]{ConcaLaurainMahadevan}.
\item {\bf Step 4:} once this limit identified, we mimick the steps of the proof of Theorem \ref{Th:RadialStability} (reduction to a small Hausdorff distance and mean value theorem for a well-chosen auxiliary function) to conclude that one necessarily has $m_\alpha=m^*_0$ for $\alpha$ small enough.
\end{itemize} 
%


\section{Proof of Theorem \ref{Th:ShapeStability}}
Throughout this section, we will denote by $\mathbb B^*$ the ball $\mathbb B(0,r^*_0)$, where $r^*_0$ is chosen so that $m^*_0=\kappa \mathds{1}_{\mathbb B^*}$ belongs to $\mathcal M_{m_0,\kappa}(\O)$.

When it makes sense, we will write $f|_{int}(y)=\lim_{x\in \B^*,x\to y}f(x)$, $f|_{ext}(y):=\lim_{x \in (\B^*)^c,x\to y}f(y)$, so that $\llbracket f\rrbracket =f|_{ext}-f|_{int}$ denotes the jump of $f$ at the boundary $\S$.

\subsection{Preliminaries}
\allowdisplaybreaks
\def\dri{{\frac{\partial}{\partial r_i}}}
\def\uk{{u_{1,\alpha}^{(k)}}}
For $\e>0$, let us introduce $\mathbb B^*_\e:=(\operatorname{Id}+\e V)\mathbb B^*$ and define $u_\e$ as the $L^2$-normalized first eigenfunction associated with $m_\e=\kappa \mathds{1}_{\mathbb B^*_\e}$.

It is well known (see e.g. \cite{Henrot2006,HenrotPierre}) that $u_\e$ expands as
\begin{equation}\label{An:EigenFunction}
u_\e=u_{0,\alpha}+\e u_{1,\alpha}+\e^2 \frac{u_{2,\alpha}}2+\underset{\e \to 0}{\operatorname{o}} (\e^2)\quad \text{in }H^1(\B^*)\text{ and in }H^1(\O \backslash \B^*),
\end{equation}
where, in particular, $u_{0,\alpha}=u_{\alpha,m^*_0}$, whereas $\lambda_\alpha(\mathbb B^*_\e)$ expands as
\begin{eqnarray}
\lambda_\alpha(\mathbb B^*_\e)&=&\lambda_{0,\alpha}+\e\lambda_\alpha'(\mathbb B^*)[V]+\frac{\e^2}{2}\lambda_\alpha''(\mathbb B^*)[V]+\underset{\e \to 0}{\operatorname{o}}(\e^2)\nonumber \\
&=&\lambda_{0,\alpha}+\e\lambda_{1,\alpha}+\frac{\e^2}{2}\lambda_{2,\alpha}+\underset{\e \to 0}{\operatorname{o}}(\e^2).\label{An:EigenValue}
\end{eqnarray}

By mimicking the proof of Lemma \ref{Le:Bounds}, one shows the following symmetry result.
\begin{lemma}\label{Le:Rad}
The function $u_{\alpha,m^*_0}$ is radially symmetric. Let $\p_{\alpha,m^*_0}$, $\tm$ and $\tilde \sigma_{\alpha,\tilde m}$ be such that $u_{\alpha,m^*_0}=\p_{\alpha,m^*_0}(|\cdot|)$, $m^*_0=\tm(|\cdot|)$ and $\tilde \sigma_{\alpha,\tilde m}=1+\alpha \tm$. Then $\p_{\alpha,m^*_0}$ satisfies the ODE
\begin{equation}\label{Eq:RadU0}\left\{
\begin{array}{ll}
-\frac{d}{dr}\left(r^{n-1}\tilde \sigma_{\alpha,\tilde m} \varphi_{\alpha,m_0^*}'\right)=\left(\lambda_\alpha(m_0^*)+\tm \right)\varphi_{\alpha,m_0^*} r^{n-1} & \text{in }(0,R)\\
\varphi_{\alpha,m_0^*}(R)=0 & 
\end{array}
\right.
\end{equation} 
complemented by the following jump conditions
\begin{equation}\label{Eq:U0}
\llbracket \p_{\alpha,m^*_0}\rrbracket (r^*_0)=\llbracket \tilde \sigma_{\alpha,m^*_0}\p_{\alpha,m^*_0}'\rrbracket (r^*_0)=0,\quad  \llbracket \tilde \sigma_{\alpha,\tilde m}\p_{\alpha,m_0^*}''\rrbracket (r^*_0)=\kappa \p_{\alpha,m^*_0}(r^*_0).
\end{equation}

Furthermore, $\p_{\alpha,m^*_0}$ converges to $\p_{0,m^*_0}$ for the strong topology of $\mathscr C^1$ as $\alpha\to 0$.
\end{lemma}
\subsection{Computation of the first and second order shape derivatives}
\paragraph{A remark on the type of vector fields we consider}
Hadamard's structure theorem (see for instance \cite[Theorem~5.9.2 and the remark below]{HenrotPierre}) ensures that the first order derivative in the direction of a vector field $V$ only depends on the normal trace of $V$. This allows us to work with only normal vector fields $V$ to compute the first order derivative. 

Once it is established that $\mathbb B^*$ is a critical shape, we can use Hadamard's structure theorem \cite[Theorem~5.9.2 and the remark below]{HenrotPierre} which states that the second order shape derivative, when computed at a critical shape only depends on the normal trace, hence we will also, for second order shape derivatives, work with normal vector fields.

 Since we are working in {two dimensions}, this means that one can deal with vector fields $V$ given in polar coordinates by 
$$
V(r^*_0,\theta)=g(\theta)\begin{pmatrix}\cos\theta\\\sin\theta \end{pmatrix}.
$$
The proof of the shape differentiability at the first and second order of $\lambda_\alpha$, based on an implicit function argument according to the method of \cite{MignotMuratPuel}, is exactly similar to \cite[Proof of Theorem 2.2]{DambrineKateb}. For this reason, we admit it. Nevertheless, in what follows, we provide some details on the computation of these derivatives for the sake of completeness, since some steps differ a bit from those done in the references above.

\paragraph{Computation and analysis of the first order shape derivative.}
Let us prove that $\B^*$ is a critical shape in the sense of \eqref{Eq:FOO}.
\begin{lemma}\label{Le:Cr}
The first order shape derivative of $\lambda_\alpha$ at $\B^*$ in direction $V$ reads 
\begin{equation}
\lambda_{1,\alpha}=\lambda_\alpha'(\B^*)[V]=\int_{\S} V\cdot \nu .
\end{equation}
For all $V \in \mathcal X(\B^*)$ (defined by \eqref{Eq:X}), one has $\lambda_{1,\alpha}=0$ meaning that $\B^*$ satisfies \eqref{Eq:FOO}. 
\end{lemma}
\begin{proof}[Proof of Lemma \ref{Le:Cr}]
First, elementary computations show that $u_{1,\alpha}$ solves
\begin{equation}\label{Eq:EigenDerivative1}
\left\{
\begin{array}{ll}
-\n \cdot \Big(\sigma_\alpha \n u_{1,\alpha}\Big)=\lambda_{1,\alpha}u_{0,\alpha}+\lambda_{0,\alpha}u_{1,\alpha}+m^*_0u_{1,\alpha}&\text{ in }\B(0,R),\\
\left\llbracket \sigma_\alpha\frac{\partial u_{1,\alpha}}{\partial \nu}\right\rrbracket (r^*_0\cos\theta ,r^*_0\sin\theta )=-\kappa g(\theta)u_{0,\alpha},& \\
\left\llbracket u_{1,\alpha}\right\rrbracket (r^*_0\cos\theta ,r^*_0\sin\theta )=-g(\theta)\left\llbracket \frac{\partial u_{0,\alpha}}{\partial r}\right\rrbracket (r^*_0\cos\theta ,r^*_0\sin\theta ),&
\end{array}
\right.
\end{equation}
where $\sigma_\alpha=1+\alpha m_0^*$ and the notation $\llbracket\cdot\rrbracket$ denote the jumps of the functions at $\S$. 
The derivation of the main equation of \eqref{Eq:EigenDerivative1} is an adaptation of the computations in \cite{DambrineKateb}.  
To derive the jump on $u_{1,\alpha}$, we follow \cite{DambrineKateb} and   differentiate the continuity equation
$\llbracket u_\e\rrbracket _{\partial \B_\e^*}=0$.
Formally plugging \eqref{An:EigenFunction} in this equation yields
$$u_{1,\alpha}\vert_{int}(r^*_0,\theta)+\left.g(\theta)\frac{\partial u_{0,\alpha}}{\partial r}\right|_{int}=u_{1,\alpha}\vert_{ext}(r^*_0,\theta)+\left.g(\theta)\frac{\partial u_{0,\alpha}}{\partial r}\right|_{ext},$$
and hence 
$$
\llbracket u_{1,\alpha}\rrbracket = u_{1,\alpha}|_{ext}-u_{1,\alpha}|_{int}=-g(\theta)\left\llbracket \frac{\partial u_{0,\alpha}}{\partial r}\right\rrbracket.
$$
Note that the same goes for the normal derivative:  we differentiate the continuity equation
$$\left\llbracket (1+\alpha m_\e)\frac{\partial u_{\e,\alpha}}{\partial \nu}\right\rrbracket_{\partial \B_\e^*}=0,$$ 
yielding
$$\left\llbracket \sigma_\alpha \frac{\partial u_{1,\alpha}}{\partial r}\right\rrbracket=-g(\theta)\left\llbracket\sigma_\alpha \frac{\partial^2 u_{0,\alpha}}{\partial r^2}\right\rrbracket.
$$
According to the equation $-\sigma_\alpha \Delta u_{0,\alpha}=\lambda_\alpha(m^*_0)u_{0,\alpha}+m^*_0u_{0,\alpha}$ in $\B^*$, this rewrites
\begin{equation}\label{Eq:Jump1}
\left\llbracket \sigma_\alpha \frac{\partial u_{1,\alpha}}{\partial r}\right\rrbracket =-\kappa g(\theta)u_{0,\alpha}.
\end{equation} 
Now, using $u_{0,\alpha}$ as a test function in \eqref{Eq:EigenDerivative1}, we get 
\begin{align*}
\lambda_{1,\alpha}&=-\int_{\mathbb B(0,R)} u_{0,\alpha} \nabla\cdot  (\sigma_\alpha\nabla u_{1,\alpha})-\int_{\mathbb B(0,R)} m^*_0 u_{1,\alpha}u_{0,\alpha}\\
&=-\int_{\mathbb B(0,R)} u_{0,\alpha} \nabla\cdot  (\sigma_\alpha\nabla u_{1,\alpha})+\int_{\mathbb B(0,R)} u_{1,\alpha} \nabla \cdot (\sigma_\alpha u_{0,\alpha})
\\&=\int_{\S}u_{0,\alpha}\left\llbracket \sigma_\alpha \frac{\partial u_{1,\alpha}}{\partial \nu} \right\rrbracket -\int_{\S}\left\llbracket \sigma_\alpha \frac{\partial u_{0,\alpha}}{\partial r}u_{1,\alpha}\right\rrbracket \\
&=-r^*_0\int_{0}^{2\pi} \kappa g(\theta)u_{0,\alpha}(r^*_0)^2\, d\theta+r^*_0\int_0^{2\pi}g(\theta)\left(\sigma_\alpha \frac{\partial u_{0,\alpha}}{\partial r}\right)\left\llbracket\frac{\partial u_{0,\alpha}}{\partial r}\right\rrbracket\, d\theta\\
&=r^*_0\int_0^{2\pi}g(\theta)\left(-\kappa u_{0,\alpha}(r^*_0)^2+\left\llbracket \sigma_\alpha\left( \frac{\partial u_{0,\alpha}}{\partial r}\right)^2\right\rrbracket \right)\, d\theta.
\end{align*}
by using that $\int_{\B(0,R)}u_{\e}^2=1$, so that $\int_{\B(0,R)} u_{0,\alpha}u_{1,\alpha}=0$ by differentiation.

Since $u_{0,\alpha}$ is radially symmetric according to Lemma \ref{Le:Rad}, we introduce the two real numbers
{\begin{equation}\label{LM}
\eta_\alpha:=-\kappa u_{0,\alpha}(r^*_0)^2+\left\llbracket \sigma_\alpha\left( \frac{\partial u_{0,\alpha}}{\partial r}\right)^2\right\rrbracket 
\quad \text{and}\quad \lambda_{1,\alpha}:=r^*_0\eta_\alpha \int_0^{2\pi}g(\theta)\, d\theta.
\end{equation}}
It is easy to see that $V$ belongs to $\mathcal X(\B^*)$ if, and only if $\int_0^{2\pi}g=0$ so that we finally have $\lambda_{1,\alpha}=0.$
\end{proof}
\paragraph{Computation of the Lagrange multiplier.}
The existence of a Lagrange multiplier $\Lambda_\alpha\in \R$ related to the volume constraint $\operatorname{Vol}(E)=m_0\operatorname{Vol}(\O)/\kappa$ is standard, and one has
$$
\forall V \in \mathcal X(\B^*), \quad \left(\lambda_\alpha'-\Lambda_\alpha\operatorname{Vol}'\right)(\B^*)[V]=0.
$$
Since 
$$\operatorname{Vol}'(\B^*)[V]=\int_{\S} V\cdot \nu =r^*_0\int_0^{2\pi}g(\theta)d\theta.$$
 (see e.g. \cite[chapitre 5]{HenrotPierre}) and since
 $$\lambda_\alpha'(\B^*)[V]=r^*_0\eta_\alpha\int_0^{2\pi}g(\theta)d\theta,$$
 where $\eta_\alpha$ is defined by \eqref{LM}, the Lagrange multiplier reads 
 $$\Lambda_\alpha=\eta_\alpha=-\kappa u_{0,\alpha}(r^*_0)^2+\left\llbracket \sigma_\alpha\left( \frac{\partial u_{0,\alpha}}{\partial r}\right)^2\right\rrbracket .$$

\paragraph{Computation of the second order derivative and second order optimality conditions.}

Let us compute the second order derivative of $\lambda_\alpha$.  
By using the Hadamard structure Theorem (see \cite[Theorem~5.9.2 and the following remark]{HenrotPierre}), since $\B^*$ is a critical shape in the sense of \eqref{Eq:FOO}, it is not restrictive to deal with vector fields that are normal to the $\partial \B^*=\mathbb S^*$, according to the so-called structure theorem which provides the generic structure of second order shape derivatives. This allow us to identify any such $V\in \mathcal X(\B^*)$ with a periodic function $g:[0,2\pi]\rightarrow \R$ such that 
$$
V(r^*_0\cos\theta ,r^*_0\sin\theta )=g(\theta)\begin{pmatrix}\cos\theta \\\sin\theta \end{pmatrix}.
$$

\begin{lemma}\label{Le:SecondDerivative}
For every $V\in \mathcal X(\B^*)$, one has  for the coefficient $\lambda_{2,\alpha}= \lambda_\alpha''(\B^*)[V,V]$ introduced in \eqref{An:EigenValue} the expression 
\begin{eqnarray*}
\lambda_{2,\alpha}&=& 2\int_{\S} \sigma_\alpha \partial_r u_{1,\alpha}|_{int}\left\llbracket\frac{\partial u_{0,\alpha}}{\partial r}\right\rrbracket V\cdot \nu -2\kappa\int_{\S}u_{1,\alpha}|_{int}u_{0,\alpha} V\cdot \nu \\
&& +\int_{\S}\left(-\frac1{r^*_0}\left\llbracket \sigma_\alpha|\n u_{0,\alpha}|^2\right\rrbracket -\frac{\kappa}{r^*_0}u_{0,\alpha}^2\right) (V\cdot \nu)^2
 -2\int_{\S}\kappa u_{0,\alpha}\left.\frac{\partial u_{0,\alpha}}{\partial r}\right|_{int} (V\cdot \nu)^2.
 \end{eqnarray*}
\end{lemma}

\begin{proof}[Proof of Lemma \ref{Le:SecondDerivative}]
In the computations below, we do not need to make the equation satisfied by $u_{2,\alpha}$ explicit, but we nevertheless will need several times the knowledge of $\llbracket u_{2,\alpha}\rrbracket $ at $\S$.
In the same fashion that we obtained the jump conditions on $u_{1,\alpha}$
Let us differentiate two times the continuity equation $\llbracket u_\e\rrbracket _{\partial \B^*_\e}=0$. We obtain
\begin{equation}\label{Eq:Jump2}
\llbracket u_{2,\alpha}\rrbracket _{\partial \B_\e^*}=-2g(\theta)\left\llbracket  \frac{\partial u_{1,\alpha}}{\partial r}\right\rrbracket -g(\theta)^2 \left\llbracket  \frac{\partial^2 u_{0,\alpha}}{\partial r^2}\right\rrbracket .
\end{equation}
Now, according to Hadamard second variation formula (see \cite[Chapitre 5, page 227]{HenrotPierre} for a proof), if $\O$ is a $\mathscr C^2$ domain and $f$ is two times differentiable at 0 and taking values in $W^{2,2}(\O)$, then one has 
\begin{equation}\label{Eq:Hada2}
  \left.\frac{d^2}{dt^2}\right|_{t=0}\int_{(\operatorname{Id}+tV)\O} f(t)=\int_{\O} f''(0)+2\int_{\partial \O} f'(0) V\cdot \nu +\int_{\partial \O}\left(H f(0)+\frac{\partial f(0)}{\partial \nu}\right) (V\cdot \nu)^2,
  \end{equation}
where $H$ denotes the mean curvature. We apply it to  $f(\e)=\sigma_{\alpha,\e}|\n u_\e|^2-m_\e u_\e^2$ on $\mathbb B(0,R)$, since  $\lambda_\alpha(m_\e)=\int_{\mathbb B(0,R)}{ f(\e)}$.
 Let us distinguish between the two subdomains $\B_\e^*$ and $(\B_\e^*)^c$. We introduce
$$
 D_1 = \left. \frac{d^2}{d\e^2}\right|_{\e=0}\int_{\B_\e^*}\left(\sigma_{\alpha,\e}|\n u_\e|^2-\kappa u_\e^2\right) \quad \text{and}\quad
 D_2 = \left. \frac{d^2}{d\e^2}\right|_{\e=0}\int_{(\B_\e^*)^c}\left(\sigma_{\alpha,\e}|\n u_\e|^2\right),
 $$
 so that $ \lambda_\alpha''(\B^*)[V,V]=D_1+D_2$.
 
  One has
  \begin{eqnarray*}
D_1 &=&\int_{\B^*}2(1+\alpha \kappa) \n u_{2,\alpha}\cdot \n u_{1,\alpha}+2\int_{\B_\e^*}(1+\alpha \kappa)|\n u_{1,\alpha}|^2
\\&&-2\kappa\int_{\B^*}u_{2,\alpha}u_{0,\alpha}-2\kappa\int_{\B^*}u_{1,\alpha}u_{0,\alpha}
\\&&+4\int_{\S}(1+\alpha \kappa) (\n u_{1,\alpha}|_{int}\cdot \n u_{0,\alpha}|_{int}) V\cdot \nu -4\kappa\int_{\S}u_{1,\alpha}|_{int}u_{0,\alpha} V\cdot \nu 
\\&&+\int_{\S}\left(\frac1{r^*_0}(1+\alpha \kappa)|\n u_{0,\alpha}|_{int}^2-\frac\kappa{r^*}u_{0,\alpha}^2+2(1+\alpha\kappa)\left.\frac{\partial u_{0,\alpha}}{\partial r}\right|_{int}\left.\frac{\partial^2 u_{0,\alpha}}{\partial r^2}\right|_{int}\right.
\\&&\left.-2\kappa u_{0,\alpha}\left.\frac{\partial u_{0,\alpha}}{\partial r}\right|_{int}\right) (V\cdot \nu)^2,
  \end{eqnarray*}
  and taking into account that the mean curvature has a sign on $(\B_\e^*)^c$, one has
  \begin{eqnarray*}
D_2 &=&\int_{(\B^*)^c}2 \n u_{2,\alpha}\cdot \n u_{1,\alpha}+2\int_{(\B^*)^c}|\n u_{1,\alpha}|^2
\\&&-4\int_{\S} (\n u_{1,\alpha}|_{ext}\cdot \n u_{0,\alpha}|_{ext})  V\cdot \nu 
\\&&+\int_{\S}\left(-\frac1{r^*_0}|\n u_{0,\alpha}|_{ext}^2-2\left.\frac{\partial u_{0,\alpha}}{\partial r}\right|_{ext}\left.\frac{\partial^2 u_{0,\alpha}}{\partial r^2}\right|_{ext}\right) (V\cdot \nu)^2.
 \end{eqnarray*}
 Summing these two quantities, we get
 \begin{eqnarray*}
 \lambda_{2,\alpha}&=&2\int_{\B(0,R)}\sigma_\alpha \n u_{0,\alpha}\cdot \n u_{2,\alpha}-2\int_{\B(0,R)}m^*_0 u_{0,\alpha}u_{2,\alpha}
 +2\int_{\B(0,R)}\sigma_\alpha|\n u_{1,\alpha}|^2-2\int_{\mathbb B(0,R)}m^*_0 u_{1,\alpha}^2
 \\&&-4\int_{\S}\sigma_\alpha \frac{\partial u_{0,\alpha}}{\partial r}\left\llbracket \frac{\partial u_{1,\alpha}}{\partial r}\right\rrbracket  V\cdot \nu -4\kappa\int_{\S}u_{1,\alpha}|_{int}u_{0,\alpha} V\cdot \nu 
 \\&&+\int_{\S}\left(-\frac1{r^*_0}\left\llbracket \sigma_\alpha|\n u_{0,\alpha}|^2\right\rrbracket -\frac{\kappa}{r^*_0}u_{0,\alpha}^2\right) (V\cdot \nu)^2
 \\&&
 -2\int_{\S}\left\llbracket \sigma_\alpha \frac{\partial u_{0,\alpha}}{\partial r}\frac{\partial^2 u_{0,\alpha}}{\partial r^2}\right\rrbracket  (V\cdot \nu)^2-2\kappa u_{0,\alpha}\left.\frac{\partial u_{0,\alpha}}{\partial r}\right|_{int} (V\cdot \nu)^2.
 \end{eqnarray*}
To simplify this expression, let us use Eq. \eqref{Eq:U0}. Introducing 
$$
D_3= \int_{\B(0,R)}\sigma_\alpha \n u_{0,\alpha}\cdot \n u_{2,\alpha}-\int_{\mathbb B(0,R)} u_{0,\alpha}u_{2,\alpha}-\lambda_\alpha(\B^*)\int_{\B(0,R)} u_{0,\alpha}u_{2,\alpha},
$$
one has
$$
D_3=\int_{\S}\llbracket u_{2,\alpha}\rrbracket \sigma_\alpha \frac{\partial u_{0,\alpha}}{\partial r},
$$
and hence, by using Equation \eqref{Eq:Jump2}, one has
 \begin{align*}
D_3 &=-2\int_{\S}\llbracket u_{2,\alpha}\rrbracket \sigma_\alpha \frac{\partial u_{0,\alpha}}{\partial r}\\
&= 4\int_{\S} \sigma_\alpha \frac{\partial u_{0,\alpha}}{\partial r}\left\llbracket \frac{\partial u_{1,\alpha}}{\partial r}\right\rrbracket V\cdot \nu +2\int_{\S} \sigma_\alpha \frac{\partial u_{0,\alpha}}{\partial r}\left\llbracket \frac{\partial^2 u_{0,\alpha}}{\partial r^2}\right\rrbracket (V\cdot \nu)^2.
\end{align*}
Similarly, let
$$
D_4=\int_{\B(0,R)}\sigma_\alpha |\n u_{1,\alpha}|^2-\int_{\B(0,R)}m^*_0u_{1,\alpha}^2.
$$
By using Eq. \eqref{Eq:EigenDerivative1} and the fact that $\lambda_{1,\alpha}=0$, one has 
\begin{eqnarray*}
D_4 &=& \lambda_\alpha(\B^*)\int_{\mathbb B(0,R)} u_{1,\alpha}^2-\int_{\S}\left\llbracket u_{1,\alpha}\sigma_\alpha \frac{\partial u_{1,\alpha}}{\partial r}\right\rrbracket 
\\&=&\lambda_\alpha(\B^*)\int_{\mathbb B(0,R)} u_{1,\alpha}^2-\int_{\S}\left\llbracket u_{1,\alpha}\right\rrbracket\left.\left(\sigma_\alpha \frac{\partial u_{1,\alpha}}{\partial r}\right)\right|_{ext}-\int_{\S}u_{1,\alpha}|_{int}\left\llbracket \sigma_\alpha \frac{\partial u_{1,\alpha}}{\partial r}\right\rrbracket 
\\&=&\lambda_\alpha(\B^*)\int_{\mathbb B(0,R)} u_{1,\alpha}^2+\int_{\S} \left(\sigma_\alpha\partial_r u_{1,\alpha}|_{ext}\left\llbracket\frac{\partial u_{0,\alpha}}{\partial r}\right\rrbracket+\kappa  u_{1,\alpha}|_{int} u_{0,\alpha}\right) V\cdot \nu.
\end{eqnarray*}
Finally, by differentiating the normalization condition $\int_{\B(0,R)}u_\e^2=1$, we get
\begin{equation}\label{normal}
\int_{\B(0,R)}u_{0,\alpha}u_{2,\alpha}+\int_{\B(0,R)}u_{1,\alpha}^2=0.\end{equation}
Combining the equalities above, one gets
\begin{eqnarray*}
\lambda_{2,\alpha}&=&2\lambda_\alpha(\B^*)\left(\int_{\B(0,R)}u_{0,\alpha}u_{2,\alpha}+\int_{\B(0,R)}u_{1,\alpha}^2\right) +4\int_{\S}  \sigma_\alpha \frac{\partial u_{0,\alpha}}{\partial r}\left\llbracket \frac{\partial u_{1,\alpha}}{\partial r}\right\rrbracket V\cdot \nu
\\&&+2\int_{\S} \sigma_\alpha \frac{\partial u_{0,\alpha}}{\partial r}\left\llbracket \frac{\partial^2 u_{0,\alpha}}{\partial r^2}\right\rrbracket (V\cdot \nu)^2+2\int_{\S} \sigma_\alpha \partial_r u_{1,\alpha}|_{ext}\left\llbracket\frac{\partial u_{0,\alpha}}{\partial r}\right\rrbracket V\cdot \nu
\\&&+2\kappa \int_{\S} u_{1,\alpha}|_{int} u_{0,\alpha}V\cdot \nu-4\int_{\S}\sigma_\alpha \frac{\partial u_{0,\alpha}}{\partial r}\left\llbracket \frac{\partial u_{1,\alpha}}{\partial r}\right\rrbracket V\cdot \nu 
 \\&&-4\kappa\int_{\S}u_{1,\alpha}|_{int}u_{0,\alpha} V\cdot \nu -\int_{\S}\left(\frac1{r^*_0}\left\llbracket \sigma_\alpha|\n u_{0,\alpha}|^2\right\rrbracket +\frac{\kappa}{r^*_0}u_{0,\alpha}^2\right) (V\cdot \nu)^2
 \\&&-2\int_{\S}\left[\sigma_\alpha \frac{\partial u_{0,\alpha}}{\partial r}\frac{\partial^2 u_{0,\alpha}}{\partial r^2}\right] (V\cdot \nu)^2-2\int_{\S}\kappa u_{0,\alpha}\left.\frac{\partial u_{0,\alpha}}{\partial r}\right|_{int} (V\cdot \nu)^2
 \\&=&2\int_{\S} \sigma_\alpha \partial_r u_{1,\alpha}|_{ext}\left\llbracket\frac{\partial u_{0,\alpha}}{\partial r}\right\rrbracket V\cdot \nu 
 -2\kappa\int_{\S}u_{1,\alpha}|_{int}u_{0,\alpha} V\cdot \nu 
 \\&& -\int_{\S}\left(\frac1{r^*_0}\left[\sigma_\alpha|\n u_{0,\alpha}|^2\right]+\frac{\kappa}{r^*_0}u_{0,\alpha}^2\right) (V\cdot \nu)^2-2\int_{\S}\kappa u_{0,\alpha}\left.\frac{\partial u_{0,\alpha}}{\partial r}\right|_{int} (V\cdot \nu)^2.
\end{eqnarray*}
We have then obtained the desired expression.
\end{proof}

\paragraph{Strong stability.}
Recall here that, as mentioned before, since we are dealing with a critical point of the functional $\lambda_\alpha$, it is enough to consider perturbation $V$ normal to the boundary of $\B^*$, in other words such that $V=(V\cdot \nu)\nu$. 
Under such an assumption, the second derivative of the volume is known to be (see e.g. \cite[Section 5.9.6]{HenrotPierre})
\begin{equation}\label{Eq:VolumeDerivative}
\operatorname{Vol}''(\B^*)[V,V]=\int_{\S}H (V\cdot \nu)^2.
\end{equation}
Hence, introducing $D_5=(\lambda_\alpha''-\eta_\alpha\operatorname{Vol}'')(\B^*)[V,V]$ and taking into account Lemma \ref{Le:SecondDerivative}, \eqref{LM} and \eqref{Eq:VolumeDerivative}, we have
\begin{eqnarray*}
D_5&=&2\int_{\S} \sigma_\alpha \partial_r u_{1,\alpha}|_{ext}\left\llbracket\frac{\partial u_{0,\alpha}}{\partial r}\right\rrbracket V\cdot \nu-2\int_{\S}\kappa u_{0,\alpha}\left.\frac{\partial u_{0,\alpha}}{\partial r}\right|_{int} (V\cdot \nu)^2
\\ &&-2\kappa\int_{\S}u_{1,\alpha}|_{ext}u_{0,\alpha} V\cdot \nu 
 +\int_{\S}\left(-\frac1{r^*_0}\left[\sigma_\alpha|\n u_{0,\alpha}|^2\right]-\frac{\kappa}{r^*_0}u_{0,\alpha}^2\right) (V\cdot \nu)^2
\\&&+\kappa \int_{\S}\frac1{r^*_0}u_{0,\alpha}^2 (V\cdot \nu)^2-\int_{\S}\frac1{r^*_0}\left[\sigma_\alpha|\n u_{0,\alpha}|^2\right] (V\cdot \nu)^2
\\&=&2\int_{\S} \sigma_\alpha \partial_r u_{1,\alpha}|_{ext}\left\llbracket\frac{\partial u_{0,\alpha}}{\partial r}\right\rrbracket V\cdot \nu-2\int_{\S}\kappa u_{0,\alpha}\left.\frac{\partial u_{0,\alpha}}{\partial r}\right|_{int} (V\cdot \nu)^2
\\ &&-2\kappa\int_{\S}u_{1,\alpha}|_{int}u_{0,\alpha} V\cdot \nu 
 -\int_{\S}\frac2{r^*_0}\left[\sigma_\alpha|\n u_{0,\alpha}|^2\right] (V\cdot \nu)^2.
\end{eqnarray*}
We are then led to determine the signature of the quadratic form
\begin{eqnarray}
\mathcal F_\alpha[V,V]&=&\frac12(\lambda_\alpha''-\Lambda_\alpha\operatorname{Vol}'')(\B^*)[V,V]  \label{Eq:Lambda2}\\
&=&\int_{\S} \sigma_\alpha \partial_r u_{1,\alpha}|_{ext}\left\llbracket\frac{\partial u_{0,\alpha}}{\partial r}\right\rrbracket V\cdot \nu -\kappa\int_{\S}u_{1,\alpha}|_{int}u_{0,\alpha} V\cdot \nu \nonumber
 \\&&+\int_{\S}\left(-\frac2{r^*}\left\llbracket \sigma_\alpha|\n u_{0,\alpha}|^2\right\rrbracket \right) (V\cdot \nu)^2-\int_{\S}\kappa u_{0,\alpha}\left.\frac{\partial u_{0,\alpha}}{\partial r}\right|_{int} (V\cdot \nu)^2.\nonumber
\end{eqnarray} 

\subsection{Analysis of the quadratic form $\mathcal{F}_\alpha$}

\paragraph{Separation of variables and first simplification.}
Each perturbation $g\in L^2(0,2\pi)$ such that $\int_0^{2\pi} g=0$ expands as 
$$
g=\sum_{k=1}^\infty \left( \gamma_k \cos(k\cdot)+\beta_k\sin(k\cdot)\right), \quad \text{with }\gamma_0=0.
$$ 
For every $k\in \N^*$, let us introduce $g_k:=\cos(k\cdot)$ and $\tilde g_k:=\sin(k\cdot)$. For any $k\in \N^*$, let $u_{1,\alpha}^{(k)}$ be the solution of Eq.~\eqref{Eq:EigenDerivative1} associated with the perturbation $g_k$.
It is readily checked that there exists a function $\varphi_{k,\alpha}:[0,R]\rightarrow \R$ such that 
$$
\forall (r,\theta)\in [0,R]\times [0,2\pi], \quad \uk(r,\theta)=g_k(\theta)z_{k,\alpha}(r).
$$
Furthermore, $\varphi_{k,\alpha}$ solves the ODE
\begin{equation}\label{Eq:PsiK} 
\left\{
\begin{array}{ll}
-\sigma_\alpha z_{k,\alpha}''-\frac{\sigma_{\alpha}}rz_{k,\alpha}'(r) =\left(\lambda_{0,\alpha}-\frac{k^2}{r^2}\right)z_{k,\alpha}+m^*_0z_{k,\alpha}&\text{ in }(0,r_0^*)\cup (r_0^*,R),\\
\left\llbracket \sigma_\alpha z_{k,\alpha}'\right\rrbracket (r^*_0)=-\kappa u_{0,\alpha}(r^*_0)& \\
\left\llbracket z_{k,\alpha}\right\rrbracket (r^*_0)=-\left\llbracket \frac{\partial u_{0,\alpha}}{\partial r}\right\rrbracket(r^*_0),
\\z_{k,\alpha}'(0)=z_{k,\alpha}(R)=0.&
\end{array}
\right.
\end{equation}
Regarding $\tilde g_k$, if we define $\tilde u_{1,\alpha}^{(k)}$ in a similar fashion, it is readily checked that 
$$\forall (r,\theta)\in [0,R]\times [0,2\pi]\, ,\tilde u_{1,\alpha}^{(k)}(r,\theta)=\tilde g_k(\theta) z_{k,\alpha}(r).$$
Therefore, any admissible perturbation $g$ writes
$$
g=\sum_{k=1}^\infty \left\{\gamma_kg_k+\beta_k\tilde g_k\right\}\quad \text{with } \gamma_0=0,
$$ 
and the solution $u_{1,\alpha}$ associated with $g$ writes
$$u_{1,\alpha}=\sum_{k=1}^\infty \left\{\gamma_k \uk+\beta_k \tilde u_{1,\alpha}^{(k)}\right\}.$$
Using the orthogonality properties of the family $\{g_k\}_{k\in \N^*}\cup \{\tilde g_k\}_{k\in \N}$, it follows that $\mathcal F_\alpha[V,V]$ given by \eqref{Eq:Lambda2} reads
\begin{eqnarray}
\mathcal F_\alpha[V,V] &=&\frac{r^*_0}2\sum_{k=1}^\infty\left(\sigma_\alpha z_{k,\alpha}'(r^*_0)|_{ext}\left\llbracket \frac{\partial u_{0,\alpha}}{\partial r}\right\rrbracket -\kappa z_{k,\alpha}|_{int}u_{0,\alpha}(r^*_0)\right)\left(\gamma_k^2+\beta_k^2\right)\nonumber
\\&&-\frac{r^*_0}2\sum_{k=0}^\infty \kappa u_{0,\alpha}(r^*_0)\frac{\partial u_{0,\alpha}}{\partial r}(r^*_0)\left(\gamma_k^2+\beta_k^2\right)-\sum_{k=1}^\infty2 \left\llbracket \sigma_\alpha |\n u_{0,\alpha}|^2\right\rrbracket \left(\gamma_k^2+\beta_k^2\right)\nonumber
\\&=&\frac{r^*_0\kappa u_{0,\alpha}(r^*_0)}2\sum_{k=1}^\infty\left(-\left.\frac{\partial u_{0,\alpha}}{\partial r}(r^*_0)\right|_{int}- z_{k,\alpha}|_{int}\right)\left(\gamma_k^2+\beta_k^2\right)\nonumber
\\&&+\sum_{k=1}^\infty \left(-2\left[\sigma_\alpha |\n u_{0,\alpha}|^2\right]+\sigma_\alpha z_{k,\alpha}'(r^*_0)|_{ext}\left\llbracket \frac{\partial u_{0,\alpha}}{\partial r}\right\rrbracket \right)\left(\gamma_k^2+\beta_k^2\right).\label{Eq:Ptn}
\end{eqnarray}
Define, for any $k\in \N$, 
$$
\omega_{k,\alpha}:=\frac{r_0^*\kappa u_{0,\alpha}(r_0^*)}{2}\left(\left.-\frac{\partial u_{0,\alpha}}{\partial r}(r^*_0)\right|_{int}- z_{k,\alpha}|_{int}(r_0^*)\right)$$ and $$ \zeta_{k,\alpha}:=-2\left\llbracket \sigma_\alpha |\n u_{0,\alpha}|^2\right\rrbracket +\sigma_\alpha z_{k,\alpha}'(r^*_0)|_{ext}\left\llbracket \frac{\partial u_{0,\alpha}}{\partial r}\right\rrbracket .
$$
Thus, 
$$
\mathcal F_\alpha[V,V]=\sum_{k=1}^\infty\left(\omega_{k,\alpha}+\zeta_{k,\alpha}\right)\left(\gamma_k^2+\beta_k^2\right).
$$
The end of the proof is devoted to proving the local shape minimality of the centered ball, which relies on an asymptotic analysis of the sequences $\{\omega_{k,\alpha}\}_{k\in \N}$ and $\{\zeta_{k,\alpha}\}_{k\in \N}$ as $\alpha$ converges to 0.
\begin{proposition}\label{Le:AAO}
There exists $C>0$ and $\overline \alpha>0$, there exists $M\in \R$ such that for any $\alpha \leq \overline \alpha$ and any $k\in \N$, one has
\begin{equation}\label{Eq:OmegaK}
\omega_{k,\alpha}\geq C>0,
\quad \text{and}\quad
\zeta_{k,\alpha}\geq - M\alpha.
\end{equation}
\end{proposition}
The last claim of Theorem \ref{Th:ShapeStability} is then an easy consequence of this proposition. The rest of the proof is devoted to the proof of Proposition~\ref{Le:AAO}, which follows from the combination of the following series of lemmas.

\begin{lemma}\label{Le:Positivity}
There exists $\overline \alpha>0$ such that, for every $\alpha\in [0, \overline \alpha]$, $z_{1,\alpha}$ is nonnegative on $(0,R)$.
\end{lemma}

\begin{proof}[Proof of Lemma~\ref{Le:Positivity}]
For the sake of notational simplicity, we temporarily drop the dependence on $\alpha$ and denote $z_{1,\alpha}$ by $z_\alpha$. The function $z_\alpha$ solves the ODE
$$
\left\{
\begin{array}{ll}
-\sigma_\alpha z_\alpha''-\frac{\sigma_{\alpha}}rz_\alpha'(r) =\left(\lambda_{0,\alpha}-\frac{1}{r^2}\right)z_\alpha+m^*z_\alpha&\text{ in }(0,r_0^*)\cup (r_0^*,R),
\\\left\llbracket \sigma_\alpha z_\alpha'\right\rrbracket (r^*_0)=-\kappa u_{0,\alpha}(r^*_0)&
\\\left\llbracket z_\alpha\right\rrbracket (r^*_0)=-\left\llbracket \frac{\partial u_{0,\alpha}}{\partial r}\right\rrbracket (r^*_0),
\\z_{\alpha}(R)=0.&
\end{array}
\right.
$$
Let us introduce  $p_\alpha=z_\alpha/u_{0,\alpha}$. One checks easily that $p_\alpha$ solves the ODE
$$
-\sigma_\alpha p_\alpha''-\frac{\sigma_{\alpha}}rp_\alpha' =-\frac{1}{r^2}p_\alpha-2p_\alpha'\frac{u_{0,\alpha}'}{u_{0,\alpha}}\quad \text{in }(0,R).
$$ 
Furthermore, $p_\alpha$ satisfies the jump conditions
$$\llbracket p_\alpha\rrbracket (r^*_0)=-\frac{\left\llbracket \partial_r u_{0,\alpha}\right\rrbracket (r^*_0)}{u_{0,\alpha}(r^*_0)} =\frac{-\alpha \kappa \partial_r u_{0,\alpha}|_{int}}{u_{0,\alpha}(r^*_0)}>0
\quad \text{and}\quad
\llbracket \sigma_\alpha p_\alpha'\rrbracket (r^*_0)=-\kappa+\frac{\sigma_\alpha \partial_r u_{0,\alpha}}{u_{0,\alpha}(r^*_0)^2}\left\llbracket\frac{\partial u_{0,\alpha}}{\partial r}\right\rrbracket.$$
To show that $z_\alpha$ is nonnegative, we argue by contradiction and consider first the case where a negative minimum is reached at an interior point $r_-\neq r^*_0$. Then, $p_\alpha$ is $\mathscr C^2$ in a neighborhood of $r_-$ and we have 
$$0\geq  -p_\alpha''(r_-)=-\frac{p_\alpha(r_-)}{\sigma_\alpha r_-^2}>0,$$
whence the  contradiction. 

To exclude the case $r_-=R$, let us notice that, according to L'Hospital's rule, one has $p_\alpha(R)=z_\alpha'(R)/u_{0,\alpha}'(R)$. According to the Hopf lemma applied to $u_{0,\alpha}$, this quotient is well-defined. If $p_\alpha(R)<0$ then it follows that $z_\alpha'(R)>0$. However, one has $p_\alpha'(r)\sim z_\alpha'(r)/(2u_{0,\alpha}(r))>0$ as $r\to R$, which contradicts the fact that a minimum is reached at $R$.

Let us finally exclude the case where $r_-=r^*_0$. Mimicking the elliptic regularity arguments used in the proofs of Lemmas \ref{Le:Bounds} and  \ref{Le:ConvergenceFP}, we get that $p_\alpha$ converges to $p_0$ as $\alpha \to 0$ for the strong topologies of $\mathscr C^0([0,r^*_0])$ and $\mathscr C^0([r^*_0,R])$.

To conclude, it suffices hence to prove that $p_0$ is positive in a neighborhood of $r^*_0$. We once again argue by contradiction and assume that $p_0$ reaches a negative minimum at $r_-\in [0,R]$. 
Notice that $r_-\neq r_0^*$ since $\llbracket p_0\rrbracket (r^*_0)=0$ and $\llbracket p_0'\rrbracket (r^*_0)=-\kappa<0$. 

If $r_-\in (0,R)$, since $r_-\neq r_0^*$, we claim that $p_0$ is $\mathscr C^2$ in a neighborhood of $r_-$ and, if $p_0(r_-)<0$, the contradiction follows from 
$$0\geq -p_0''(r_-)=-\frac{p_0(r_-)}{(r_-)^2}>0.$$ For the same reason, a negative minimum cannot be reached at $r=0$.

If $r_-=R$, we observe that $p_0(R)=z_0'(R)/u_{0,0}'(R)$.
According to the Hopf lemma applied to $u_{0,0}$, this quantity is well-defined. If $p_0(R)<0$, then it follows that $z_0'(R)>0$. However, $p_0'(r)\sim z_0'(r)/(2u(r))>0$ as $r\to R$, which contradicts the fact that $R$ is a minimizer.

Therefore $p_0$ is positive in a neighborhood of $r^*_0$ and we infer that $p_\alpha$ is non-negative, so that, in turn, $z_\alpha \geq 0$ in $[0,R]$.
\end{proof}

\begin{lemma}\label{Cl:Comp}
Let $\overline \alpha$ be defined as in Lemma \ref{Le:Positivity}. Then, for every $ \alpha\in [0, \overline \alpha]$ and every $k\in \N$, 
\begin{equation}\label{Eq:CompPs}z_{k,\alpha}\leq z_{1,\alpha}.\end{equation}
As a consequence, for any $\alpha \leq \overline \alpha$ and any $k\in \N$, there holds $\omega_{k,\alpha}\geq \omega_{1,\alpha}$.
\end{lemma}
\begin{proof}[Proof of Lemma \ref{Cl:Comp}]
Since $\omega_{k,\alpha}-\omega_{1,\alpha}=\frac{r_0^*\kappa u_{0,\alpha}(r_0^*)}{2}\left(- z_{k,\alpha}|_{int}(r^*_0)+z_{1,\alpha}|_{int}(r^*_0)\right)$, and since we further have $\frac{r_0^*\kappa u_{0,\alpha}(r_0^*)}{2}>0$, 
the fact that $\omega_{k,\alpha}\geq \omega_{1,\alpha}$ will follow from \eqref{Eq:CompPs}, on which we now focus. Let us set $\Psi_k=z_{1,\alpha}-z_{k,\alpha}$. From the jump conditions on $z_{1,\alpha}$ and $z_{k,\alpha}$, one has $\llbracket \Psi_k\rrbracket (r_0)=\llbracket \sigma_\alpha\Psi_k'\rrbracket (r_0)=0$.
The function $\Psi_k$ satisfies
\begin{eqnarray*}
-\sigma_\alpha\Psi_k''-\sigma_\alpha\frac{\Psi_k'}r&=&-\left(\lambda_{0,\alpha}-\frac{k^2}{r^2}\right)z_{k,\alpha}-m^*_0z_{k,\alpha}+\left(\lambda_{0,\alpha}-\frac1{r^2}\right)z_{1,\alpha}+m^*_0\psi_{1,\alpha}
\\&>&\left(\lambda_{0,\alpha}-\frac{k^2}{r^2}\right)z_{k,\alpha}-m^*_0z_{k,\alpha}+\left(\lambda_{0,\alpha}-\frac{k^2}{r^2}\right)z_{1,\alpha}+m^*_0\psi_{1,\alpha}
\\&>&\left(\lambda_{0,\alpha}-\frac{k^2}{r^2}\right)\Psi_{k}+m^*_0\Psi_{k}.
\end{eqnarray*}
since $z_{1,\alpha}\geq0$, according to Lemma~\ref{Le:Positivity}. Since $\Psi_k$ satisfies Dirichlet boundary conditions, $\Psi_k\geq0$ in $(0,R)$. 
\end{proof}

\begin{lemma}\label{Cl:Final}
There exists $C>0$ such that, for every $\alpha \in [0, \overline \alpha]$, where $\overline \alpha$ is introduced on Lemma~\eqref{Le:Positivity}, one has $\omega_{1,\alpha}\geq C$.
\end{lemma}
\begin{proof}[Proof of Lemma \ref{Cl:Final}]
Let us introduce $\Psi=-{\partial u_{0,\alpha}}/{\partial r}- z_{1,\alpha}$. Since
$$\omega_{1,\alpha}=\frac{r_0^*\kappa u_{0,\alpha}(r_0^*)}{2}\Psi(r_0^*)$$ and since $\frac{r_0^*\kappa u_{0,\alpha}(r_0^*)}{2}$ converges, as $\alpha \to0$, to $\frac{r_0^*\kappa u_{0,0}(r_0^*)}{2}>0$, it suffices to prove that $\Psi(r_0^*)\geq C>0$ for some $C$ when $\alpha \to 0$.  
According to \eqref{Eq:RadU0}, we have $\llbracket \Psi\rrbracket (r^*_0)=\llbracket \Psi'\rrbracket (r^*_0)=0$. Furthermore, 
$\Psi(R)=-\frac{\partial u_{0,\alpha}}{\partial r}(R)>0$ according to the Hopf Lemma and $\Psi(0)=0$. Finally, since $\Psi$ solves the ODE, one has
$$
-\frac1r(\sigma_\alpha \Psi')'=\left(\lambda_{0,\alpha}-\frac1{r^2}\right)\Psi+m^*\Psi\quad \text{in }(0,R),
$$
it follows that $\Psi$ is positive in $(0,R]$. Furthermore, $\Psi$ converges to $\Psi_0$ for the strong topology of $\mathscr C^0([0,R])$ and $\Psi_0$ solves the ODE
$$
\left\{\begin{array}{ll}
-\frac1r( \Psi_0')'=\left(\lambda_{0,0}-\frac1{r^2}\right)\Psi_0+m^*_0\Psi_0 & \text{in }(0,R)\\
\Psi_0(R)=-\frac{\partial u_{0,0}}{\partial r}(R)>0. & 
\end{array}\right.
$$ 
Hence there exists $C>0$ such that, for every $\alpha \in [0, \overline \alpha]$, one has $\Psi(r^*_0)\geq C>0$.
\end{proof}

It remains to prove the second inequality of\eqref{Eq:OmegaK}. 
As a consequence of the convergence result stated in Lemma~\ref{Le:Rad}, one has
\begin{equation}\label{bus:1904}
\left\llbracket \sigma_\alpha |\n u_{0,\alpha}|^2\right\rrbracket =\operatorname{O}(\alpha), \quad \left\llbracket\frac{\partial u_{0,\alpha}}{\partial r}\right\rrbracket=\left.\alpha \kappa \frac{\partial u_{0,\alpha}}{\partial r}\right|_{int}<0.
\end{equation}
It follows that we only need to prove that there exists a constant $M>0$ such that, for any $\alpha\in [0, \overline \alpha]$, and any $k\in \N^*$, 
\begin{equation}\label{Eq:Tain3}
M\geq \sigma_\alpha z_{k,\alpha}'|_{ext}(r^*_0)
\end{equation}
so that
$$
\zeta_{k,\alpha}=\operatorname{O}(\alpha)+\sigma_\alpha z_{k,\alpha}'\vert_{ext}(r^*_0)\left.\alpha \kappa \frac{\partial u_{0,\alpha}}{\partial r}(r^*_0)\right|_{int}\geq \operatorname{O}(\alpha)-M\alpha\kappa\left|\left.\frac{\partial u_{0,\alpha}}{\partial r}(r^*_0)\right|_{int}\right|
$$ 

To show the estimate \eqref{Eq:Tain3}, let us distinguish between small and large values of $k$.
To this aim, we introduce $N\in \N$ as he smallest integer such that 
\begin{equation}\label{tain}
\lambda_{0,\alpha}+m^*_0-\frac{k^2}{r^2}<0 \text{ in }(0,R)
\end{equation} 
for every $k\geq N$ and $\alpha \in [0, \overline \alpha]$. The existence of such an integer follows immediately from the convergence of $(\lambda_{0,\alpha})_{\alpha>0}$ to $\lambda_0(m^*_0)$ as $\alpha\to 0$. 

\medskip

First, we will prove that, for every $k\geq N$,
\begin{equation}\label{Eq:DerPos}
z_{k,\alpha}'(r^*_0)|_{ext}<0
\end{equation} 
and that there exists $M>0$ such that, for every $k\leq N$, 
\begin{equation}\label{Eq:DerPos2}
|z_{k,\alpha}'(r^*_0)|_{ext}|\leq M
\end{equation} 
which will lead to \eqref{Eq:Tain3} and thus yield the desired conclusion. 

To show \eqref{Eq:DerPos}, let us argue by contradiction, assuming that $z_{k,\alpha}'(r^*_0)|_{ext}>0$.
Since the jump $\llbracket \sigma_\alpha z_{k,\alpha}'\rrbracket =-\kappa u_{0,\alpha}(r^*_0)$ is negative, it follows that 
$$
(1+\alpha \kappa)z_{k,\alpha}'(r^*_0)|_{int}=z_{k,\alpha}'(r^*_0)|_{ext}-\llbracket \sigma_\alpha z_{k,\alpha}'\rrbracket >0.
$$
By mimicking the reasonings in the proof of Lemma~\ref{Le:Positivity}, $z_{k,\alpha}$ cannot reach a negative minimum on $(0,r^*_0)$ since \eqref{tain} holds true. Therefore, since $z_{k,\alpha}(0)=0$ and $z_{k,\alpha}'(r^*_0)|_{int}>0$, one has necessarily $z_{k,\alpha}(r^*_0)|_{int}>0$, which in turn gives 
$z_{k,\alpha}(r^*_0)|_{ext}>0$ since $\llbracket z_{k,\alpha}\rrbracket =-\alpha \kappa \frac{\partial u_{0,\alpha}}{\partial r}>0$.
 
Furthermore, $z_{k,\alpha}(R)=0$. Since $z_{k,\alpha}(r^*_0)|_{ext}>0$ and $z_{k,\alpha}'(r^*_0)|_{ext}>0$, it follows that $z_{k,\alpha}$ reaches a positive maximum at some interior point $r_1$, satisfying hence
$$
0\leq -z_{k,\alpha}''(r_-)=\left(\lambda_{0,\alpha}+m^*-\frac{k^2}{r^2}\right)z_{k,\alpha}(r_-)<0,
$$ 
leading to a contradiction.

\medskip

Let us now deal with small values of $k$, by assuming $k\leq N$. We will prove that \eqref{Eq:DerPos2} holds true. 
To this aim, we will compute $z_{k,\alpha}$. Let $J_k$ (resp. $Y_k$) be the $k$-th Bessel function of the first (resp. the second) kind. One has 
$$z_{k,\alpha}(r)=\left\{\begin{array}{ll}
A_{k,\alpha}J_k(\sqrt{\frac{\lambda_{0,\alpha}+\kappa}{1+\alpha\kappa}} \frac{r}R) & \text{ if }r\leq r^*_0, \\
B_{k,\alpha}J_k(\sqrt{\lambda_{0,\alpha}}\frac{r}R)+C_{k,\alpha} Y_k(\sqrt{\lambda_{0,\alpha}}\frac{r}R) & \text{ if }r^*_0\leq r\leq R,
\end{array}\right.$$
where $X_{k,\alpha}=(B_{k,\alpha},C_{k,\alpha},A_{k,\alpha})$ solves the linear system
$$
\mathcal A_{k,\alpha}X_{k,\alpha}=b_\alpha
$$
where
$$
b_\alpha=\begin{pmatrix}0\\-\kappa u_{0,\alpha}(r^*_0)\\-\left[\partial_R u_{0,\alpha}\right]\\\end{pmatrix}
$$
and
$$
\mathcal A_{k,\alpha}=\begin{pmatrix}
J_k\left(\sqrt{\lambda_0}\right)&Y_k\left(\sqrt{\lambda_0}\right)&0
\\\sqrt{\lambda_{0,\alpha}}J_k'(\sqrt{\lambda_{0,\alpha}+\kappa}\r)&\sqrt{\lambda_{0,\alpha}}Y_k'(\sqrt{\lambda_{0,\alpha}+\kappa}\r)&-\sqrt{\frac{\lambda_{0,\alpha}+\kappa}{1+\alpha\kappa}}J_k'(\sqrt{\frac{\lambda_{0,\alpha}+\kappa}{1+\alpha\kappa}}\r)
\\J_k(\sqrt{\lambda_{0,\alpha}}\r)&Y_k(\sqrt{\lambda_{0,\alpha}}\r)&-J_k(\sqrt{\frac{\lambda_{0,\alpha}+\kappa}{1+\alpha\kappa}}\r)
\end{pmatrix}.
$$
It is easy to check that
$$
\Vert \mathcal A_{k,\alpha}-\mathcal A_{k,0}\Vert \leq M\alpha
$$ 
where $M$ only depends\footnote{Indeed, $\{J_k,Y_k\}_{k\leq N}$  are uniformly bounded in $\mathscr C^2([r_0^*/R-\e,R])$ for every $\e>0$ small enough. Since we consider a finite number of indices $k$, there exists $\delta>0$ (depending only on $N$) such that 
$$
\forall k\in \{0, \dots,N\}, \quad \det(\mathcal A_{k,\alpha})\geq \delta>0.
$$ 
Then, since $\Vert X_\alpha-X_0\Vert \leq M\alpha$, it follows from the Cramer formula that there exists $M$ (depending only on $N$) such that 
$\Vert 
X_{k,\alpha}-X_{k,0}\Vert_{L^\infty}\leq M\alpha.
$ 
} on $N$.
Hence it is enough to prove that 
$|\psi_{k,0}'(r^*_0)|\leq M$ for some $M>0$ depending only on $N$, which is straightforward since the set of indices is finite. The expected conclusion follows.

\subsection{Conclusion}
From Eq. \eqref{Eq:OmegaK} and Lemma \ref{Cl:Final}, there exists $C>0$ and $M>0$ such that $\omega_{k,\alpha}\geq C>0$ and $\zeta_{k,\alpha}\geq -M\alpha$ for every $\alpha\in [0, \overline \alpha]$ and $k\in \N$, from which we infer that 
\begin{align*}
\mathcal F_\alpha[V,V]&\geq \left(C-M\alpha\right)\sum_{k=1}^\infty \left(\gamma_k^2+\beta_k^2\right)\geq \frac{C}2 \Vert V\cdot \nu \Vert_{L^2}^2.
\end{align*}
according to Eq. \eqref{Eq:Ptn}. 

\subsection{Concluding remark: possible extension to higher dimensions}\label{Se:CclSha}
Let us briefly comment on possible extensions of this method to higher dimensions. Indeed, although we do not tackle this issue in this article, we believe that the coercivity norm obtained in Theorem~\ref{Th:ShapeStability} could also be obtained in the three-dimensional case. Nevertheless, we believe that such an extension would need tedious and technical computations.
Since our objective here was to introduce a methodology to investigate stability issues for the shape optimization problems we deal with, we slightly comment on this claim and explain how we believe that our proof can be adapted to the case $d=3$. 

Let $\O$ denote the ball $\mathbb B(0,R)$ in $\R^3$ and $\B^*$ be the centered three-dimensional ball $\mathbb B(0,R)$ of volume $m_0|\O|/\kappa$. Let us assume without loss of generality that $R=1$, so that $\partial \B^*$ is the euclidean unit sphere $\mathbb S^2$. 

As a preliminary result, one first has to show that the principal eigenfunction $u_{\alpha,m^*_0}$ is radially symmetric and that $\B^*$ is a critical shape by the same arguments as in the proof of Theorem~\ref{Th:ShapeStability}, which allows us to compute the Lagrange multiplier $\Lambda_\rho$ associated to the volume constraint. Let $\mathcal L_{\Lambda_\rho}$ be  the associated shape Lagrangian.

For an integer $k$, we define $H_k$ as the space of spherical harmonics of degree $k$ i.e as the eigenspace associated with the eigenvalue $-k(k+1)$ of the Laplace-Beltrami operator $\Delta_{\mathbb S^2}$. $H_k$ has  finite dimension $d_k$, and we furthermore have
$$L^2(\mathbb S^2)=\bigoplus_{k=1}^\infty H_k.$$
Let us consider a Hilbert basis $\{y_{k,\ell}\}_{\ell=1,\dots,d_k}$ of $H_k$. 

For an admissible vector field $V,$ one must then expand $ V\cdot \nu$ in the basis of spherical harmonics as 
\begin{equation}\label{eq:fin}
V\cdot \nu=\sum_{k=1}^\infty\sum_{\ell=1}^{d_k} \alpha_{k,\ell}(V\cdot \nu) y_{k,\ell}.\end{equation}

Then, one has to diagonalise the second-order shape derivative  of $\mathcal L_{\Lambda_\rho}$ and prove that there exists a sequence of coefficients $\{\omega_{k,\ell,\rho}\}_{k\in \N\,, 0\leq \ell\leq d_k}$ such that for every $V\cdot \nu$ expanding as \eqref{eq:fin}, the second order derivative of the shape Lagrangian in direction $V$ reads 
$$
\mathcal L_{\Lambda_\rho}''=\sum_{k=1}^\infty\sum_{\ell=1}^{d_k}\left(\alpha_{k,\ell}(V\cdot\nu)\right)^2 \omega_{k,\ell}.
$$
We believe this diagonalization can be proved using separation of variables and the orthogonality properties of the family $\{y_{k,\ell}\}_{k\in \N^*,\ell=1,\dots,d_k}$.

Using the separation of variables, each coefficient $\omega_{k,\ell}$ can be written in terms of derivatives of a family of solutions of one dimensional differential equations. The main difference with the proof of Theorem~\ref{Th:ShapeStability} comes from the fact that the main part of the ODE is not $-\frac1r\frac{d}{dr}(r(1+\alpha m_0^*)\frac{d}{dr})$ anymore, but $-\frac1{r^2}\frac{d}{dr}(r^2(1+\alpha m_0^*)\frac{d}{dr})$. The important fact is that maximum principle arguments may still be used to analyze the diagonalized expression of $\mathcal L_{\Lambda_\rho}''$ and to obtain a uniform bound from below for the sequence $\{\alpha_{k,\ell}\}_{k\in \N^*\,, \ell=1,\dots,d_k}$.

\appendix
\section{Proof of Lemma~\ref{diff:valPvecP}}\label{Ap:Differentiability}

We prove hereafter that the mapping $m\mapsto(u_{\alpha,m}, \lambda_\alpha(m))$ is twice differentiable (and even $\mathscr C^\infty$) in the $L^2$ sense, the proof of the differentiability in the weak $W^{1,2}(\O)$ sense being similar. 
Let $m^*\in \mathcal M_{m_0,\kappa}(\O)$, $\sigma_\alpha:=1+\alpha m^*$,  and $(u_0,\lambda_0)$ be the eigenpair associated with $m^*$. Let $ h\in \mathcal{T}_{m^*}$ (see Def.~\ref{def:tgtcone}). 
Let $m^*_h:=m^*+h$ and $\sigma_{m^*+h}:=1+\alpha(m^*+h).$ Let $(u_{h},\lambda_h)$ be the eigenpair associated with $m^*_h$. 
Let us introduce the mapping $G$ defined by
$$
G:\left\{\begin{array}{ll}
\mathcal{T}_{m^*}\times W^{1,2}_0(\O)\times \R\to W^{-1,2}(\O)\times \R,&
\\(h, v, \lambda)\mapsto \left(-\n \cdot (\sigma_{m^*+h}\n v))-\lambda v-m^*_hv, \int_\O v^2-1\right).&\end{array}
\right.$$
From the definition of the eigenvalue, one has $G(0,u_{0},\lambda_{0})=0$.
Moreover, $G$ is $\mathscr C^\infty$ in $\mathcal{T}_{m^*}\cap B\times W^{1,2}_0(\O)\times \R$, where $B$ is an open ball centered at $ 0$. The differential of $G$ at $(0,u_0,\lambda_0)$ reads
$$
D_{v,\lambda}G(0,u_0,\lambda_0)[w,\mu]=\left(-\n \cdot (\sigma_\alpha \n w)-\mu u_0-\lambda_0 w-m^*w, \int_\O 2u_0w\right).
$$
Let us show that this differential is invertible. We will show that, if $(z,k)\in W^{-1,2}(\O)\times \R$, then there exists a unique pair  $(w,\mu)$ such that $D_{v,\lambda}G(0,u_0,\lambda_0)[w,\mu]=(z,k)$.
According to the Fredholm alternative, one has necessarily $\mu=-\langle z,u_0\rangle_{L^2(\O)}$ and for this choice of $\mu$, there  exists a solution $w_1$ to the equation 
$$
-\n \cdot (\sigma_\alpha \n w)-\mu u_0-\lambda_0 w-m^*w=z\quad \text{in }\O .
$$ 
Moreover, since $\lambda_0$ is simple, any other solution is of the form $w=w_1+tu_0$ with $t\in \R$. From the equation $2\int_\O u_0w=k$, we get $t=k/2-\int_\O w_1u_0$. Hence, the pair $(w,\mu)$ is uniquely determined. According to the implicit function theorem, the mapping $h\mapsto (u_h,\lambda_h)$ is $\mathscr C^\infty$  in a neighbourhood of $\vec 0$.

{\small 
\bibliographystyle{abbrv}
\nocite{*}
\bibliography{BiblioPrinc}
\addcontentsline{toc}{part}{Bibliography}
}
\end{document}